\documentclass[twoside,11pt]{article}

\usepackage{blindtext}

\usepackage{jmlr2e}

\usepackage{color}
\usepackage{amssymb}     
\usepackage{amsmath}        
\usepackage{arydshln}
\usepackage{subfigure}
\usepackage{epsfig}
\usepackage{hyperref}    
\usepackage{graphicx}    
\usepackage{algpseudocode}
\usepackage{algorithmicx,algorithm}
\usepackage{colortbl}
\usepackage{tabularx}
\usepackage{booktabs}       
\usepackage{makecell}        
\usepackage{threeparttable}   
\usepackage{multirow}        
\usepackage{cases}           
\usepackage{pifont}          
\usepackage{adjustbox}       
\usepackage{multicol}
\usepackage{longtable}
\usepackage{algorithm}   
\usepackage{algorithmicx}
\usepackage{algpseudocode}
\usepackage{lipsum}
\newcommand\blfootnote[1]{%
	\begingroup
	\renewcommand\thefootnote{}\footnote{#1}%
	\addtocounter{footnote}{-1}%
	\endgroup
}
\newtheorem{assumption}{Assumption}
\makeatletter

\newcommand{\Rmnum}[1]{\expandafter\@slowromancap\romannumeral #1@}
\makeatother
\usepackage{lastpage}
\jmlrheading{  }{2023}{1-\pageref{LastPage}}{7/20; Revised  /  }{ /  }{  -    }{Meixuan He, Yuqing Liang, Jinlan Liu and Dongpo Xu}

\ShortHeadings{Convergence of Adam for Non-convex Objectives}{He, Liang, Liu, and Xu}
\firstpageno{1}

\begin{document}

\title{Convergence of Adam for Non-convex Objectives: Relaxed Hyperparameters and Non-ergodic Case}

\author{\name Meixuan He$^{\dagger}$ 
	\email hemx833@nenu.edu.cn\\
	\name Yuqing Liang$^{\dagger}$ 
	\email liangyq337@nenu.edu.cn
	\\
	\name Jinlan Liu 
	\email liujl627@nenu.edu.cn\\
	\name Dongpo Xu$^{\ast}$
	\email xudp100@nenu.edu.cn\\
	\addr School of Mathematics and Statistics\\
	Northeast Normal University\\
	Changchun 130024, China}

\editor{My editor}
\blfootnote{$\dagger$ Equal Contribution.}
\blfootnote{ $\ast$ Corresponding Author.}
\maketitle

\begin{abstract}%   <- trailing '%' for backward compatibility of .sty file
Adam is a commonly used stochastic optimization algorithm in machine learning. However, its convergence is still not fully understood, especially in the non-convex setting. This paper focuses on exploring hyperparameter settings for the convergence of vanilla Adam and tackling the challenges of non-ergodic convergence related to practical application. The primary contributions are summarized as follows: firstly, we introduce precise definitions of ergodic and non-ergodic convergence, which cover nearly all forms of convergence for stochastic optimization algorithms. Meanwhile, we emphasize the superiority of non-ergodic convergence over ergodic convergence. Secondly, we establish a weaker sufficient condition for the ergodic convergence guarantee of Adam, allowing a more relaxed choice of hyperparameters. On this basis, we achieve the almost sure ergodic convergence rate of Adam, which is arbitrarily close to $o(1/\sqrt{K})$. More importantly, we prove, for the first time, that the last iterate of Adam converges to a stationary point for non-convex objectives. Finally, we obtain the non-ergodic convergence rate of $O(1/K)$ for function values under the Polyak-\L ojasiewicz (PL) condition. These findings build a solid theoretical foundation for Adam to solve non-convex stochastic optimization problems.
\end{abstract}

\begin{keywords}
Adam,  stochastic optimization, non-ergodic convergence, PL condition, non-convex optimization
\end{keywords}

\section{Introduction}
This paper focuses on the stochastic optimization problem of the following form
\begin{equation}\label{wenti1}
\min_{\mathbf{x}\in\mathbb{R}^d}\left\lbrace f(\mathbf{x}):=\mathbb{E}_{\xi\sim \mathbb{P}}[\ell(\mathbf{x};\xi)]\right\rbrace ,
\end{equation}
where $\xi$ is a random variable obeying an unknown distribution $\mathbb{P}$, denoting a randomly selected data sample or random noise. 
\par To solve the optimization problem \eqref{wenti1}, \citet{robbins1951stochastic} proposed the stochastic gradient descent (SGD) algorithm employing a subsample of the full gradient at each iteration, known as a minibatch. Specifically, the iteration of SGD is 
\begin{align}
\mathbf{x}^{k+1}=\mathbf{x}^k-\eta_k\mathbf{g}^k,\nonumber
\end{align} 
where $\mathbf{g}^k$ denotes the stochastic gradient and $\eta_k$ is the step size. In particular, SGD exhibits efficacy in solving large-scale machine learning problems \citep[e.g.,][]{bottou2010large,bottou2018optimization,shen2019deep}. However, since the stochastic gradient merely serves as an unbiased estimator of the full gradient, only the sublinear convergence rate of $O(1/K)$ is attained in the strongly convex setting \citep[see][]{agarwal2009information, rakhlin2012making,luo2022sgd}. Meanwhile, this estimation also brings a certain variance, which may produce oscillations in the objective function \citep{ruder2016overview} or cause the algorithm to converge to local minima \citep{10.5555/2969033.2969154}.
\par In recent years, many adaptive step size  algorithms have been developed to address the above issues. As a pioneering work, \citet{duchi2011adaptive} proposed AdaGrad, which adaptively adjusts step size based on the accumulation of historical information (the square of the gradient values). Despite its excellent performance in handling sparse data \citep{dean2012large,pennington2014glove}, AdaGrad loses its edge when the gradients are dense. Subsequently, RMSProp \citep{tieleman2012lecture} adopted an exponential moving average technique to integrate historical information of gradient. By assigning more weight to the current gradient, RMSProp overcomes the monotonically decreasing of the step size in AdaGrad. Given this, \citet{kingma2014adam} introduced the first-order momentum in RMSProp to modify the update direction  and developed the Adam algorithm. For online convex problems, the convergence of Adam with diminishing step sizes has been provided by \citet{kingma2014adam}. However, \citet{reddi2019convergence} pointed out a flaw in the proof of \citet{kingma2014adam} and constructed a convex optimization problem where Adam exhibits divergence. Therefore, some attempts have focused on designing variants of the Adam algorithm to guarantee convergence. \citet{defazio2022adaptivity} introduced the MADGRAD method to update in a sparse manner based on the double averaging of AdaGrad. By the maximum operation, \citet{reddi2019convergence} proposed AMSGrad to solve the non-monotonicity of the effective step size in Adam. Further, \citet{Chen2018ClosingTG} introduced Padam to achieve better generalization performance, which expands the range of the adaptive parameter in AMSGrad. Moreover, \citet{chen2019convergence} set $\theta_{k}=1-1/k$ in Adam to achieve a decreasing effective step size and proposed the AdaFom algorithm. In particular, the second-order momentum in AdaFom is exactly the average of squared historical gradients. In addition, \citet{luo2018adaptive} proposed AdaBound by clipping the second-order momentum, effectively solving the problem of step sizes vanishing and exploding. 

\par Although the proof of \citet{kingma2014adam} has some drawbacks, the superior empirical performance of Adam makes it the most popular algorithm at present. This motivates us to explore the convergence behavior of vanilla Adam. \citet{zhang2022adam} claimed that vanilla Adam can achieve convergence and argued that the divergent example in  \citet{reddi2019convergence} is inconsistent with practice. The difference is that \citet{reddi2019convergence} first set the hyperparameters and then choose a $\beta$-dependent optimization problem, whereas \citet{zhang2022adam} set the hyperparameters according to the optimization problem, which is more in line with the actual scenario. So, the argument made by \citet{zhang2022adam} does not contradict the counterexample presented by \citet{reddi2019convergence}. When the second-order momentum parameter is close enough to $1$, \citet{zhang2022adam} demonstrated that Adam converges to the neighborhood of a critical point and converges to a critical point if the strong growth condition is satisfied. For Adam and AdaGrad, \citet{defossezsimple} provided a simple proof of convergence and derived an upper bound on the norm of the gradient. In addition, \citet{wang2022provable} conducted an analysis of Adam on $(L_0,L_1)$-smooth objective functions. On the other hand, some recent work \citep{chen2019convergence,zou2019sufficient} aimed to provide the sufficient condition to guarantee the convergence of Adam. By introducing the generalized second-order momentum estimate, \citet{chen2019convergence} proposed a unified framework of Adam-type methods and provided a sufficient condition for the convergence of Generalized Adam, but this condition is difficult to check. In the non-convex setting, \citet{chen2019convergence} obtained the convergence rate of $O(\ln K/\sqrt{K})$ for AMSGrad and AdaFom. In contrast, \citet{zou2019sufficient} presented an easy-to-check sufficient condition, abbreviated as \textbf{SC-Zou}, for the convergence of Adam. Since \textbf{SC-Zou} only depends on the step size $\eta_{k}$ and the second-order momentum parameter $\theta_ {k}$, \citet{zou2019sufficient} contributed to understanding the convergence behavior and hyperparameter settings of Adam. 
\par To highlight the contribution of this paper, we perform a comparison of our work with the previous literature. In the convex case, 
\citet{kingma2014adam,reddi2019convergence} obtained an $O(1/\sqrt{K})$ convergence rate of Adam. For non-convex objectives,  \citet{zou2019sufficient,chen2019convergence,guo2021novel,defossezsimple,zhang2022adam} showed that Adam converges with the rate of $O(\ln K/\sqrt{K})$. In contrast, we provide the almost sure convergence analysis of Adam, with the rate arbitrarily close to $o(1/\sqrt{K})$. Note that the above results focus on ergodic convergence, but the output of the last iterate is more relevant with practice, so the non-ergodic convergence results considered in this paper are more meaningful. In particular, Proposition \ref{contopro} shows that the sufficient condition for Adam established in Corollary \ref{suffcor}, abbreviated as \textbf{SC-Adam}, is  weaker than \textbf{SC-Zou}. This indicates that \textbf{SC-Adam} offers a more flexible hyperparameter selection for Adam (The weaker the condition, the larger the selection range of hyperparameters). Table \ref{table:results} summarizes our main contributions and innovations.

\begin{table}[h]
\renewcommand{\arraystretch}{1.2}
\tabcolsep=0.5mm
\centering
\scalebox{0.81}{
\begin{adjustbox}{center}
\begin{threeparttable}
\begin{tabular}{c c c c c c}
	\toprule
	\textbf{Optimizer}  
	& \textbf{Setting} 
	& \textbf{\makecell{Non-ergodic \\ convergence}} 
	& \textbf{\makecell{almost sure\\ convergence}} & \textbf{\makecell{Ergodic convergence
	\\ Sufficient condition }}
	\\
	\midrule
	Adam
	\citep{kingma2014adam}
	& convex 
	& \ding{55} 
	& \ding{55} 
	& \ding{55} 
	\\
	\midrule
	Generic Adaptive Method \citep{reddi2019convergence} 
	& convex 
	& \ding{55} 
	& \ding{55} 
	& \ding{55} 
	\\
	\midrule
	Generalized Adam \citep{chen2019convergence} 
	& non-convex 
	& \ding{55} 
	& \ding{55} 
	& \checkmark 
	\\
	\midrule
	Adam \citep{defossezsimple} 
	& non-convex 
	& \ding{55} 
	& \ding{55} 
	& \ding{55} 
	\\
	\midrule
	Adam-Style Algorithm \citep{guo2021novel} 
	& non-convex 
	& \ding{55} 
	& \ding{55} 
	& \ding{55} 
	\\
	\midrule
	Adam \citep{zhang2022adam} 
	& non-convex 
	& \ding{55} 
	& \ding{55}
	&\ding{55}
	\\
	\midrule
	Generic Adam \citep{zou2019sufficient} 
	& non-convex 
	& \ding{55} 
	& \ding{55} 
	& \checkmark \\
	\midrule
	\textbf{Adam\,(ours)}  
	& non-convex 
	& \checkmark 
	& \checkmark 
	& \checkmark \\
	\bottomrule
\end{tabular} 
\end{threeparttable}
\end{adjustbox}}
\caption{Comparison with previous works.}
\label{table:results}
\end{table}

\par Specifically, this paper establishes the following new convergence results for Adam: 
\begin{itemize}
\item We present precise definitions of ergodic and non-ergodic convergence and clarify their relationship in Section \ref{dingyi}, showing that non-ergodic convergence is stronger than ergodic convergence. In particular, we demonstrate that the convergence of minimum and uniform output is a special case of ergodic convergence.
	
\item We propose a more relaxed sufficient condition for the convergence of Adam in Corollary \ref{suffcor}, referred to as \textbf{SC-Adam}. Proposition \ref{contopro} shows that \textbf{SC-Adam} is more relaxed than \textbf{SC-Zou}, and there is no additional monotonicity restriction on the second-order momentum $\theta_{k}$ and allows a wider range of the step size $\eta_{k}$.
		
\item We establish the first almost sure convergence rate analysis of Adam in Theorem \ref{th6} and Corollary \ref{cor4}, showing that the rate of gradient norm is arbitrarily close to $o(1/\sqrt{K})$.
	
\item We first prove the non-ergodic convergence of Adam in the non-convex setting, that is, $\lim_{k\rightarrow\infty} \mathbb{E}[\lVert \nabla f(\mathbf{x}^k)\rVert]=0$ and $\lim_{k\rightarrow\infty} \lVert\nabla f(\mathbf{x}^k)\rVert=0\ a.s.$ (see Theorem \ref{th2}). Additionally, we present Example \ref{e3} to show the superiority of non-ergodic convergence compared to existing minimum and uniform convergence.

\item Under the PL condition, we achieve the non-ergodic convergence of Adam for function values in Theorem \ref{th5}, that is, $\lim_{k\rightarrow\infty} \mathbb{E}[f(\mathbf{x}^k)]=f^*$ and $\lim_{k\rightarrow\infty} f(\mathbf{x}^k)=f^*\ a.s.$ Furthermore, Theorem \ref{th3} shows that by choosing the step size $\eta_k=\frac{\sqrt{M^2+\epsilon}}{(1-\beta)v}\cdot\frac{1}{k+1}$, the non-ergodic convergence rate of $O(1/K)$ is attained.
\end{itemize}
\par The rest of this paper is organized as follows: Section \ref{Prelim} provides the notations, definitions, and assumptions used in this paper, along with the Adam algorithm. In Section \ref{ERGODIC}, we analyze the ergodic convergence rate of Adam. We establish the non-ergodic convergence results for Adam in Section \ref{Non-ergodic Convergence of Adam}. Finally, we conclude this paper in Section \ref{sec:conclusion}.

\section{Preliminaries}\label{Prelim}
In this section, we first introduce notations in Section \ref{fuhao},  then provide definitions of ergodic and non-ergodic convergence in Section \ref{dingyi}. The assumptions used in our theoretical analysis are given in Section \ref{jiashe}. Section \ref{suanfa} presents the pseudo-code of the Adam algorithm.
\subsection{Notations}\label{fuhao}
In this paper, the vector operations are all element-wise, which means for any $\epsilon$, $p>0$ and $\mathbf{x}$, $\mathbf{y}\in\mathbb{R}^d$, we have  $\mathbf{x}\pm\epsilon:=(\mathbf{x}_1\pm\epsilon,\mathbf{x}_2\pm\epsilon,\cdots,\mathbf{x}_d\pm\epsilon)\in\mathbb{R}^d$, $(\mathbf{x})^p:=((\mathbf{x}_1)^p,(\mathbf{x}_2)^p,\cdots ,(\mathbf{x}_d)^p)\in\mathbb{R}^d$, 
$\mathbf{y}\mathbf{x}:=(\mathbf{y}_1\mathbf{x}_1,\mathbf{y}_2\mathbf{x}_2,\cdots,\mathbf{y}_d\mathbf{x}_d)\in\mathbb{R}^d$, $\mathbf{y}/\mathbf{x}:=(\mathbf{y}_1/\mathbf{x}_1,\mathbf{y}_2/\mathbf{x}_2,\cdots,\mathbf{y}_d/\mathbf{x}_d)\in\mathbb{R}^d$. For clarity, Table \ref{table:2} provides other notations employed in this study.
%\onecolumn
\renewcommand{\arraystretch}{1.2}
\begin{longtable}{p{1.5cm}p{13cm}}
		\toprule
		\textbf{Symbol}
		&\multicolumn{1}{c}
		{\textbf{Meaning}}
		\\
		\hline
		$\mathbb{R}^d$ 
		&The $d$-dimensional real coordinate space
		\\
		\hline
		$\mathbb{N}_{+}$ 
		&The set of positive integers
		\\
		\hline
		$\mathbf{g}^k$ 
		&The stochastic gradient, $\mathbf{g}^k:=\nabla\ell(\mathbf{x}^k;\xi^k)$\\
		\hline
		$\mathcal{F}^k$
		&The $\sigma$-algebra,  $\mathcal{F}^k:=\sigma(\mathbf{x}^1,\mathbf{x}^2,\cdots,\mathbf{x}^k)$
		\\
		\hline
		$\mathbf{x}_i^k$ 
		&The $i$-th component of $\mathbf{x}^k$ obtained in the $k$-th iteration
		\\
		\hline
		$\beta_k$ 
		&The first-order momentum  parameter
		\\
		\hline
		$\theta_k$ 
		&The second-order momentum  parameter
		\\
		\hline
		$\eta_k$ 
		&The iteration step size
		\\
		\hline
		$v$ 
		&The coefficient of PL condition
		\\
		\hline
		$M$ 
		&Upper bound on the gradient norm
		\\
		\hline
		$\left\|  \mathbf{x} \right\|  $ &${l}_2$-norm of the vector $\mathbf{x}$
		\\
		\hline
		$\left\|  \mathbf{x} \right\|  _1$
		&${l}_1$-norm of the vector $\mathbf{x}$
		\\
		\hline
		$\nabla{f(\mathbf{x})}$
		&The gradient of function $f$ at $\mathbf{x}$
		\\
		\hline
		$\nabla_j{f(\mathbf{x})}$
		&The $j$-th component of the gradient $\nabla{f(\mathbf{x})}$
		\\
		\hline
		$\Theta$
		&If $c_1\le\left| f(\mathbf{x})/g(\mathbf{x})\right|\le c_2 $ for some $c_1,\ c_2>0$, then we say that $f(\mathbf{x})=\Theta(g(\mathbf{x}))$
		\\
		\hline
		$O$
		&If $\left| f(\mathbf{x})/g(\mathbf{x})\right|\le c $ for some $c>0$, then we say that $f(\mathbf{x})=O(g(\mathbf{x}))$
		\\
		\hline
		$o$
		&If $\lim_{\mathbf{x}\rightarrow\mathbf{x}_{0}}\left| f(\mathbf{x})/g(\mathbf{x})\right|=0 $ for a certain  $\mathbf{x}_{0}$, then we say that $f(\mathbf{x})=o(g(\mathbf{x}))$
		\\
		\bottomrule
		\\
     \caption{Table of notations.}
	\label{table:2}
\end{longtable}

\subsection{Definitions}\label{dingyi}
In this section, we revisit the ergodic and non-ergodic convergence results of stochastic optimization algorithms in previous work and provide precise definitions. Through illustrative examples and a theoretical proposition, we investigate the relationship between ergodic and non-ergodic convergence, emphasizing the superiority of our non-ergodic results.

\begin{definition}[Ergodic convergence]\label{ergodic}
Let $(x_{n})_{n\geq 1}$ be a sequence of real numbers and $(\omega_{n,k})_{n\geq 1, 1\leq k\leq n}$ be a double array of real numbers that satisfies $\omega_{n,k}\geq 0$ and  $\sum_{k=1}^{n}\omega_{n,k}=1$ for any $n\in \mathbb{N}_{+}$. The ergodic sequence of $x_{k}$ with respect to $\omega_{n,k}$ is defined by $\bar{x}_{n}:=\sum_{k=1}^{n}\omega_{n,k}x_{k}$. If 
\begin{align*}
	\lim_{n\to\infty}
	\bar{x}_{n}=x,
\end{align*}
then the convergence of $x_{n}$ to $x$ is said to be ergodic. 
\end{definition}

\begin{definition}[Non-ergodic convergence]\label{nonergodic}
Let $(x_{n})_{n\geq 1}$ be a sequence of real numbers, if 
\begin{align*}
	\lim_{n\to\infty}x_{n}=x,
\end{align*}
then the convergence of $x_{n}$ to $x$ is said to be non-ergodic.
\end{definition}
\par In particular, ergodic convergence describes the relationship between the average of the sequence $(x_{k})_{1\leq k\leq n}$ and the limit $x$. In contrast, non-ergodic convergence focuses on the difference of $x_{n}$ from the limit $x$. An intuition derived from Definitions \ref{ergodic} and \ref{nonergodic} is that the non-ergodic convergence of $x_{n}$ is a stronger result that is more consistent with practice. To support this claim, we first give a simple counter-example to demonstrate that ergodic convergence does not necessarily imply the non-ergodic counterpart.
\begin{example}
Let $x_{k}=\lvert \sin \big((k\pi)/2\big)\rvert$ and $\omega_{n,k}=1/n$, then the ergodic sequence of $x_{k}$ is 
\begin{align*}
	\bar{x}_{n}
	:=
	\sum_{k=1}^{n}\omega_{n,k} 
	x_{k}
	=\dfrac{1}{n}
	\sum_{k=1}^{n}
	\Big\lvert \sin\,\dfrac{k\pi}{2}
	\Big\rvert.
\end{align*}
Then when $n$ is an odd number, we have
\begin{align*}
	\bar{x}_{n}
	=\dfrac{1}{n}
	\sum_{k=1}^{n}
	\Big\lvert \sin\,\dfrac{k\pi}{2}\Big\rvert
	=\dfrac{1}{n}
	\left(\dfrac{n-1}{2}+1\right)
	=\dfrac{n+1}{2n},
\end{align*}
which indicates that $\bar{x}_{n}$ converges to $1/2$ when $n$ is odd. On the other hand, when $n$ is an even number, we have 
\begin{align*}
	\bar{x}_{n}
	=\dfrac{1}{n}
	\sum_{k=1}^{n}
	\Big\lvert \sin\,\dfrac{k\pi}{2}\Big\rvert
	=\dfrac{1}{n}
	\cdot
	\dfrac{n}{2}
	=\dfrac{1}{2}.
\end{align*}
Thus, we obtain that $\bar{x}_{n}\to 1/2\,(n\to\infty)$, that is, $x_{n}$ converges to $1/2$ in ergodicity. Note that the odd sequence of $(x_{n})_{n\geq 1}$ converges to $1$, but the even sequence of $(x_{n})_{n\geq 1}$ converges to $0$, so the limit of such sequence does not exist, that is, $x_{n}$ fails to converge to $1/2$ in non-ergodic sense.
\end{example}
\par Convergence analysis is essential in understanding the performance of various stochastic optimization algorithms. In particular,  \citet{chambolle2016ergodic,ward2020adagrad,huang2021super} focused on the behaviors of ergodic sequences. On the other hand,  \citet{sun2019non,sun2020novel,he2022revisit,LIU202327} worked on achieving non-ergodic convergence. This paper establishes the relationship between ergodic and non-ergodic convergence in Proposition \ref{prop-non}, accompanied with a rigorous proof. Specifically, we consider $\omega_{n,k}$ with a limit of zero and show that non-ergodic convergence outperforms ergodic convergence.
\begin{proposition}\label{prop-non}
Suppose that $(x_{n})_{n\geq 1}$ is a sequence of real numbers and $(\omega_{n,k})_{n\geq 1, 1\leq k\leq n}$ is a double array of real numbers such that $\lim_{n\to\infty}\omega_{n,k}=0$ for any $1\leq k\leq n$. Then if $\lim_{n\to\infty}x_{n}=x$, we have 
\begin{align*}
	\lim_{n\to\infty}
	\bar{x}_{n}
	=x.
\end{align*}
where $\bar{x}_{n}$ is the ergodic sequence of $x_{k}$ with respect to $\omega_{n,k}$, that is, $\bar{x}_{n}:=\sum_{k=1}^{n}\omega_{n,k}x_{k}$.
\end{proposition}
\begin{proof}
See Appendix \ref{A}.
\end{proof}
\par Note that the condition $\omega_{n,k}\to 0\,(n\to\infty)$ for any $1\leq k\leq n$ in Proposition \ref{prop-non} is easily to be satisfied in most existing ergodic convergence analyses, see the following examples.
\begin{example}\label{e2}
According to different types of optimization problems, we explain the rationality of the condition that $\lim_{n\to\infty}\omega_{n,k}=0$ for any $1\leq k\leq n$, thereby the ergodic convergence can be deduced from the non-ergodic convergence.
\begin{enumerate}
\item[$(1)$]Strongly Convex Setting
\par \citet{guo2020revisiting} revisited the average strategy in stochastic algorithms, and the sequence to measure the performance of SGD can be viewed as
\begin{align*}
	\bar{y}_{n}
	:=
	\dfrac{\sum_{k=1}^{n}k^{\alpha}
	\mathbb{E}\big[
	\lVert\mathbf{x}^{k}-\mathbf{x}^{*}
	\rVert^{2}\big]}
	{\sum_{k=1}^{n}k^{\alpha}},
\end{align*}
where $\alpha\geq 0$ is a constant and $\mathbf{x}^{*}$ is a minimizer of the function $f$, that is, $\mathbf{x}^{*}=\arg\min_{\mathbf{x}\in\mathbb{R}^d}f(\mathbf{x})$. By choosing $\omega_{n,k}:=k^{\alpha}/\big(\sum_{k=1}^{n}k^{\alpha}\big)$ and $y_{k}:=\mathbb{E}[\lVert\mathbf{x}^{k}-\mathbf{x}^{*}\rVert^{2}]$, we have
\begin{align*}
	\bar{y}_{n}
	=\sum_{k=1}^{n}
	\omega_{n,k}y_{k}
	\quad 
	\text{and}\quad 
	\sum_{k=1}^{n}\omega_{n,k}
	=\sum_{k=1}^{n}
	\dfrac{k^{\alpha}}
	{\sum_{k=1}^{n}k^{\alpha}}=1,
	\quad \forall\,n\geq 1.
\end{align*}
Thus, $\bar{y}_{n}$ is an ergodic sequence of $y_{k}$ with respect to $\omega_{n,k}$. Since $\alpha\geq 0$, then it follows from the integral test inequality that 
\begin{align*}
	\sum_{k=1}^{n}k^{\alpha}
	\geq  \int_{0}^{n}x^{\alpha}\,dx
	=\dfrac{x^{1+\alpha}}{1+\alpha}
	\,\bigg\vert_{0}^{n}
	=\dfrac{n^{1+\alpha}}{1+\alpha}.
\end{align*}
Hence, for any $1\leq k\leq n$, we have
\begin{align*}
	0\leq 
	\dfrac{k^{\alpha}}
	{\sum_{k=1}^{n}k^{\alpha}}
	\leq 
	\dfrac{n^{\alpha}}
	{\sum_{k=1}^{n}k^{\alpha}}
	\leq n^{\alpha}\,\dfrac{1+\alpha}{n^{1+\alpha}}
	=\dfrac{1+\alpha}{n}.
\end{align*}
Then we apply Squeeze theorem in conjunction with $\lim_{n\to\infty}(1+\alpha)/n=0$ to get 
\begin{align*}
	\lim_{n\to\infty}\omega_{n,k}
	=\lim_{n\to\infty}
	\dfrac{k^{\alpha}}{\sum_{k=1}^{n}
		k^{\alpha}}
	=0,\quad
	\forall\,1\leq k\leq n,
\end{align*}
which satisfies the condition of $\omega_{n,k}$ in Proposition \ref{prop-non}. Thus, we have
\begin{align*}
	\lim_{n\to\infty}
	y_{n}=0
	\quad \Longrightarrow\quad 
	\lim_{n\to\infty}
	\bar{y}_{n}=0.
\end{align*}

\item[$(2)$]Convex Setting
\par \citet{sebbouh2021almost} well studied the almost sure convergence of SGD, and the sequence for measuring the behavior of SGD can be regarded as
\begin{align*}
	\bar{y}_{n}
	:=
	\dfrac{\sum_{k=1}^{n}\eta_{k}
	\big(f(\mathbf{x}^{k})-f^{*}\big)}
	{\sum_{k=1}^{n}\eta_{k}},
\end{align*}	
where $\eta_{k}$ is the step size and $f^{*}$ is the minimum value of $f(\mathbf{x})$. Then by setting $\omega_{n,k}:=\eta_{k}/\big(\sum_{k=1}^{n}\eta_{k}\big)$ and $y_{k}:=f(\mathbf{x}^{k})-f^{*}$, we have 
\begin{align*}
	\bar{y}_{n}
	=\sum_{k=1}^{n}
	\omega_{n,k}y_{k}
	\quad \text{and}\quad 
	\sum_{k=1}^{n}\omega_{n,k}=
	\sum_{k=1}^{n}
	\dfrac{\eta_{k}}
	{\sum_{k=1}^{n}\eta_{k}}=1,\quad 
	\forall\,n\geq 1,	
\end{align*}
which indicates that $\bar{y}_{n}$ is the ergodic sequence of $y_{k}$. In particular, \citet{sebbouh2021almost} required the step size $\eta_{k}$ to satisfy $\sum_{k=1}^{\infty}\eta_{k}=\infty$ and $\sum_{k=1}^{\infty}\eta_{k}^{2}<\infty$ for the convergence of $\bar{y}_{n}$. Then we know that $\eta_{k}$ is convergent and therefore bounded, that is, there exists a constant $C\geq 0$ such that $\eta_{k}\leq C$. Thus, for any $1\leq k\leq n$, 
\begin{align*}
	0\leq \omega_{n,k}
	=\dfrac{\eta_{k}}
	{\sum_{k=1}^{n}\eta_{k}}
	\leq \dfrac{C}
	{\sum_{k=1}^{n}\eta_{k}},
\end{align*}
Using Squeeze theorem in conjunction with $\lim_{n\to\infty}C/\big(\sum_{k=1}^{n}\eta_{k}\big)=0$ $(\text{since}\; \sum_{k=1}^{\infty}\eta_{k}=\infty)$, we arrive at
\begin{align*}
	\lim_{n\to\infty}
	\omega_{n,k}
	=\lim_{n\to\infty}
	\dfrac{\eta_{k}}
	{\sum_{k=1}^{n}\eta_{k}}
	=0,\quad \forall\,1\leq k\leq n,
\end{align*}
which fulfills the condition of Proposition \ref{prop-non}. Therefore, we have
\begin{align*}
	\lim_{n\to\infty}
	y_{n}=0
	\quad \Longrightarrow\quad 
	\lim_{n\to\infty}
	\bar{y}_{n}=0.
\end{align*}
\item[$(3)$] Non-convex Setting
\par The following standard average of the gradient norm is often used to measure the performance of stochastic adaptive algorithms, such as Yogi \citep{zaheer2018adaptive} and Super-Adam \citep{huang2021super}.
\begin{align*}
	\bar{y}_{n}
	:=\dfrac{1}{n}
	\sum_{k=1}^{n}
	\mathbb{E}[\lVert\nabla f(\mathbf{x}^{k})\rVert^{2}].
\end{align*}
Let's take $\omega_{n,k}:=1/n$ and $y_{k}:=\mathbb{E}[\lVert\nabla f(\mathbf{x}^{k})\rVert^{2}]$, then
\begin{align*}
	\bar{y}_{n}
	=\sum_{k=1}^{n}\omega_{n,k}
	y_{k}\quad 
	\text{and}\quad 
	\sum_{k=1}^{n}
	\omega_{n,k}
	=\sum_{k=1}^{n}\dfrac{1}{n}
	=1,\quad \forall\,n.
\end{align*}
Then it follows from Definition \ref{ergodic} that $\bar{y}_{n}$ is the ergodic sequence of $y_{k}$. Upon since 
\begin{align*}
	\lim_{n\to\infty}\omega_{n,k}
	=\lim_{n\to\infty}
	\dfrac{1}{n}=0,\quad \forall\,1\leq k\leq n,
\end{align*}
then we can confirm that the condition of $\omega_{n,k}$ in Proposition \ref{prop-non} is satisfied. Thus, 
\begin{align*}
	\lim_{n\to\infty}
	y_{n}=0
	\quad \Longrightarrow\quad 
	\lim_{n\to\infty}
	\bar{y}_{n}=0.
\end{align*}
\end{enumerate}
\end{example}
\par Next, we present two other convergence forms in Example \ref{e3}, referred to as minimum and uniform convergence. Specifically, the minimum and uniform convergence satisfy the definition of ergodic convergence and can be obtained from non-ergodic convergence.

\begin{example}\label{e3}
We discuss the convergence of minimum and uniform output, which is usually used to measure the behavior of stochastic optimization algorithms.
\begin{enumerate}
\item[$(1)$]Minimum Convergence \citep{chen2019convergence,khaled2022better}
\par The sequence of the minimum output is defined as
\begin{align*}
	\bar{y}_{n}
	:=
	\min_{1\leq k\leq n}
	\mathbb{E}[\lVert \nabla f(\mathbf{x}^{k})\rVert^{2}].
\end{align*}
Assume that $m\in[1, n]$ is the index that minimizes the norm of the gradient, that is, $\mathbb{E}[\lVert\nabla f(\mathbf{x}^{m})\rVert^{2}]=\min_{1\leq k\leq n}\mathbb{E}[\lVert\nabla f(\mathbf{x}^{k})\rVert^{2}]$. Let's set $y_{k}:=\mathbb{E}[\lVert\nabla f(\mathbf{x}^{k})\rVert^{2}]$ and
\begin{align*}
	\omega_{n,k}
	:=
	\begin{cases}
		1,&k=m\\
		0,&k\neq m
	\end{cases}.
\end{align*}
Then the following equalities hold.
\begin{align*}
	\bar{y}_{n}
	=\sum_{k=1}^{n}
	\omega_{n,k}y_{k}
	\quad \text{and}
	\quad 
	\sum_{k=1}^{n}
	\omega_{n,k}=
	\omega_{n,m}
	=1,\quad 
	\forall\,n\geq 1,
\end{align*}
which implies that $\bar{y}_{n}$ is the ergodic sequence of $y_{k}$. Thus, the convergence of $y_{n}$ is said to be ergodic. Meanwhile, by the fact that $\lVert\nabla f(\mathbf{x}^{k})\rVert^{2}\geq 0$, we can get 
\begin{align*}
	0\leq \bar{y}_{n}
	=\min_{1\leq k\leq n}
	y_{k}
	\leq y_{n}.
\end{align*}
This indicates that we can obtain $\lim_{n\to\infty}\bar{y}_{n}=0$ from $\lim_{n\to\infty}y_{n}=0$.

\item[$(2)$] Uniform Convergence \citep{zou2018weighted,liu2022hyper}
\par The sequence of the uniform output is described as follows.
\begin{align*}
	\bar{y}_{n}
	:=
	\mathbb{E}\left[\left\| \nabla f(\mathbf{x}^{\tau})\right\|^2\right],
\end{align*}
where $\tau$ is randomly and uniformly chosen from the set $\left\lbrace 1,2,\cdots,n\right\rbrace$. By choosing $y_{k}:=\mathbb{E}[\lVert\nabla f(\mathbf{x}^{k})\rVert^{2}]$ and
\begin{align*}
	\omega_{n,k}
	:=
	\begin{cases}
		1,&k=\tau\\
		0,&k\neq \tau
	\end{cases},
\end{align*}
we can derive 
\begin{align*}
	\bar{y}_{n}
	=\sum_{k=1}^{n}
	\omega_{n,k}y_{k}
	\quad \text{and}
	\quad 
	\sum_{k=1}^{n}
	\omega_{n,k}=
	\omega_{n,\tau}
	=1,\quad 
	\forall\,n\geq 1.
\end{align*}
This means that $\bar{y}_{n}$ is the ergodic sequence of $y_{k}$, and the convergence of $y_{n}$ is said to be ergodic. Since $\tau$ is uniformly chosen from $\{1, 2, \cdots, n\}$, then we have
\begin{align*}
	\bar{y}_{n}
	:=\mathbb{E}\left[\left\| \nabla f(\mathbf{x}^{\tau})\right\|^2\right]
	=\frac{1}{n}\sum_{k=1}^{n}
	\mathbb{E}\left[\left\| \nabla f(\mathbf{x}^k)\right\|^2\right],
\end{align*}	
Thus, using the same technique as Example \ref{e2} with $\alpha=0$, we can acquire  $\lim_{n\to\infty}\bar{y}_{n}=0$ from $\lim_{n\to\infty}y_{n}=0$.
\end{enumerate}
\end{example}

\subsection{Assumptions}\label{jiashe}
In this section, we formalize the assumptions used in our convergence analysis.
\begin{assumption}\label{Assunbia}
The stochastic gradient is an unbiased estimation of the full gradient,  that is, $\mathbb{E}\left[ \mathbf{g}^k|\mathcal{F}^k\right] =\nabla f(\mathbf{x}^k)$, and there exists $M>0$ such that $\left\| \mathbf{g}^k\right\| \le M\ a.s.$
\end{assumption}

\begin{assumption}\label{Assmin}
The function $f$ is lower bounded with $f^*>-\infty$ and is $L$-smooth, that is,  for $\forall\,\mathbf{x},\mathbf{y}\in{\mathbb{R}^d}$, there exists $L>0$, such that $\ \left\| \nabla f(\mathbf{x})-\nabla f(\mathbf{y})\right\|  \le L\left\| \mathbf{x}-\mathbf{y}\right\| $.
\end{assumption}

\par It should be noted that the unbiasedness and boundedness of the stochastic gradient in Assumption \ref{Assunbia} are commonly used in the analysis of stochastic optimization algorithms \citep[e.g.,][]{Nemirovski2008RobustSA,rakhlin2012making,sun2020novel,dong2022stochastic,defazio2022adaptivity}. In addition, Assumption \ref{Assmin} implies the differentiability of $f$, but $f$ can be non-convex. Note that the $L$-smoothness of $f$ is a standard assumption in the non-convex analysis of a class of Adam algorithms \citep[e.g.,][]{chen2019convergence,gadat2022asymptotic,kavis2022high,wang2023convergence}.

\subsection{Adam Algorithm}\label{suanfa}
In this section, we give the pseudo-code of Adam in Algorithm \ref{Gadafom}, where the parameter $\epsilon\geq 0$ is for numerical stability, $\beta_{k}$ and $\theta_{k}$ denote the first-and second-order momentum parameters, respectively.
In particular, if $\beta_k=0$ in Algorithm \ref{Gadafom}, then Adam becomes  RMSProp \citep[see][]{tieleman2012lecture}. Moreover, when $\theta_k=1-1/k$, Algorithm \ref{Gadafom} is equal to AdaFom proposed by \citet{chen2019convergence}. Meanwhile, the fourth step in Algorithm \ref{Gadafom} reveals that each component of $\mathbf{v}^k$ is non-negative for any $k=1,2,\cdots$.
\begin{algorithm}[H]
	\caption{Adam}
	\label{Gadafom}
	\begin{algorithmic}[1]
		\Require
		$\eta_k>0,\ 
		\epsilon\ge0,\ 0\le\beta_k\le\beta<1,\ \theta_k\in(0,1)$
		\Ensure
		$\mathbf{m}^0=\mathbf{0},\mathbf{v}^0=\mathbf{0},\mathbf{x}^1\in\mathbb{R}^d$
		\For{$k=1,2,\cdots$}
		\State{sample the stochastic gradient $\mathbf{g}^k$}
		\State{$\mathbf{m}^k=\beta_k\mathbf{m}^{k-1}+(1-\beta_k)\mathbf{g}^k$}
		\State{$\mathbf{v}^k=\theta_k\mathbf{v}^{k-1}+(1-\theta_k)(\mathbf{g}^k)^2$}
		\State{$\mathbf{x}^{k+1}=\mathbf{x}^k-\eta_k\frac{\mathbf{m}^k}{\sqrt{{\mathbf{v}}^k+\epsilon}}$}
		\EndFor
	\end{algorithmic}
\end{algorithm}

\section{Ergodic Convergence of Adam}\label{ERGODIC}
In this section, we perform the ergodic  convergence rate analysis of Adam. In particular, Sections \ref{Minimum Output} and \ref{Uniform Output and Sufficient Condition} establish the convergence of the minimum and uniform output, respectively. Further, we achieve the almost sure convergence rate of Adam in Section \ref{Almost Sure Convergence}.

\subsection{Minimum Convergence}\label{Minimum Output}
In this section, we present the convergence rate for Adam with minimum output.
\begin{theorem}[Minimum convergence]\label{th1}
Suppose that Assumptions \ref{Assunbia} and \ref{Assmin} are satisfied, and $(\alpha_k)_{k\ge 1}$ is a non-increasing real sequence. Let $(\mathbf{x}^k)_{k\ge 1}$ be the sequence generated by Algorithm \ref{Gadafom} with $\eta_k=\Theta(\alpha_{k})$, that is, there exist  positive constants $C_0$ and $\tilde{C}_0$ such that $C_0\alpha_k\le \eta_{k}\le \tilde{C}_0 \alpha_k$. Then for any $K\ge 1$, we have
\begin{equation}\label{th1.1}
	\begin{aligned}
	\min_{1\le k\le K}\mathbb{E}\left[\left\| \nabla f(\mathbf{x}^k)\right\|^2\right]
	\le
	\frac{C_1
		+C_2\sum_{k=1}^{K}\eta_k(1-\theta_k)+C_3 \sum_{k=1}^{K}\eta_{k}^2}{\sum_{k=1}^{K}\eta_{k}},\nonumber
	\end{aligned}
\end{equation}
where $C_1:=\frac{\sqrt{M^2+\epsilon}\left( f(\mathbf{x}^1)-f^*\right) }{1-\beta}$, $C_2:=\frac{\tilde{C}_0 M^4\sqrt{d(M^2+\epsilon)}}{\epsilon^{\frac{3}{2}}C_0(1-\beta)^2}$ and  $C_3:=\frac{2\tilde{C}_0^2 M^2L\sqrt{M^2+\epsilon}}{\epsilon C_0^2(1-\beta)^2}$ are all finite values.
\end{theorem}
\begin{proof}
	See Appendix \ref{Proof of th1}.
\end{proof}

\par It is worth mentioning  that our work mitigates the dependence of the convergence rate for the Adam algorithm on dimension $d$. Specifically, previous studies by \citet{zou2019sufficient,defossezsimple,zhang2022adam} have achieved the convergence rate of $O(\ln K/\sqrt{K})$, but relied on $O(d)$. In contrast, Theorem \ref{th1} attains a significant improvement by reducing the dependency to $O(\sqrt{d})$. Meanwhile, the condition that $\eta_{k}=\Theta(\alpha_{k})$ with a non-increasing $\alpha_{k}$ relaxes the requirement of decreasing step sizes. A similar assumption can be found in \cite{wang}, showing that the step size satisfying $m/k\leq \eta_{k}\leq M/k$ leads to the optimal convergence rate of SGD.

\begin{corollary}\label{Ergcor}
Suppose that the conditions in Theorem \ref{th1} still hold. By choosing $\theta_k=1-1/k$, and $\eta_k=1/k^q$ for any $0<q\le 1$, we have 
\begin{equation}\label{equation2.49}
    \begin{aligned}
	\min_{1\le k\le K}\mathbb{E}\left[\left\| \nabla f(\mathbf{x}^k)\right\|^2\right]=
	\begin{cases}
	O(1/K^{q}) ,
	& \text{$0<q< \frac{1}{2}$}\\
	O(\ln K/\sqrt{K}), 
	& \text{$q=\frac{1}{2}$}\\
	O(1/K^{1-q}) ,
	& \text{$\frac{1}{2}<q<1$}\\
	O(1/\ln K) ,
	& \text{$q=1$}\\
\end{cases}.\nonumber
\end{aligned}
\end{equation}
\end{corollary}
\par Note that \citet{defossezsimple} obtained the convergence rate of $O(\ln K/\sqrt{K})$ by setting $\theta_k=1-1/K$ and $\eta_k=1/\sqrt{K}$. However, the total number of iterations $K$ remains unknown because it usually varies with the desired accuracy. In contrast, we achieve an $O(\ln K/\sqrt{K}) $ convergence rate by choosing $\theta_k=1-1/k$ and $\eta_k=1/\sqrt{k}$, which are independent of the total number $K$. Moreover, our theoretical analysis is established on arbitrary $K\ge 1$. Nevertheless, \citet{defossezsimple} required $K>\beta/(1-\beta)$, which is large for the standard value of $\beta=0.9$.

\begin{corollary}\label{Ergcor3}
Suppose that the conditions in Theorem \ref{th1} still hold, if we set $\theta_k=1-1/k^{p}$ for any $p>0$ and $\eta_k=1/\sqrt{k}$, then
\begin{equation}
	\begin{aligned}
	\min_{1\le k\le K}\mathbb{E}\left[\left\| \nabla f(\mathbf{x}^k)\right\|^2\right]=
	\begin{cases}
	O(1/K^{p}) ,
	& \text{$0<p< \frac{1}{2}$}\\
	O(\ln K/\sqrt{K}), 
	& \text{$\frac{1}{2}\le p$}\\
	\end{cases}.\nonumber
	\end{aligned}	
\end{equation}
\end{corollary}

\par In particular, Corollaries \ref{Ergcor} and  \ref{Ergcor3} demonstrate that Adam can attain the optimal rate of $O(\ln K/\sqrt{K})$ for a wide range of hyperparameter values.

\subsection{Uniform Convergence and Sufficient Condition}\label{Uniform Output and Sufficient Condition}
In this section, we analyze the performance of Adam with uniform output and establish a more relaxed sufficient condition for convergence in Corollary \ref{suffcor}.

\begin{theorem}[Uniform convergence]\label{th4}
Suppose that Assumptions \ref{Assunbia} and \ref{Assmin} are satisfied and $(\alpha_k)_{k\ge 1}$ is a non-increasing real sequence. Let $(\mathbf{x}^k)_{k\ge 1}$ be the sequence generated by Algorithm \ref{Gadafom} with $\eta_k=\Theta(\alpha_{k})$, that is, there exist  positive constants $C_0$ and $\tilde{C}_0$ such that $C_0\alpha_k\le \eta_{k}\le \tilde{C}_0 \alpha_k$. Then for an integer $\tau$ chosen randomly and uniformly from the set $\left\lbrace 1,2,\cdots,K\right\rbrace $ and any $K\ge 1$, we have
\begin{equation}\label{th4.1}
\begin{aligned}
	\mathbb{E}\left[\left\| \nabla f(\mathbf{x}^{\tau})\right\|^2\right]
	\le\frac{ C'_1
		+C'_2\sum_{k=1}^{K}\eta_k{(1-\theta_k)}+C'_3 \sum_{k=1}^{K}\eta_{k}^2}{K\eta_{K}},\nonumber
\end{aligned}
\end{equation}
where  $C'_1:=\frac{\tilde{C}_0\sqrt{M^2+\epsilon}\left( f(\mathbf{x}^1)-f^*\right) }{(1-\beta)C_0}$, $
C'_2:=\frac{\tilde{C}_0^2 M^4\sqrt{d(M^2+\epsilon)}}{\epsilon^{\frac{3}{2}}C_0^2(1-\beta)^2}$ and $
C'_3:=\frac{2\tilde{C}_0^3 M^2L\sqrt{M^2+\epsilon}}{\epsilon C_0^3(1-\beta)^2}$ are all finite values.
\end{theorem}
\begin{proof}
See Appendix \ref{Proof of th4}.
\end{proof}	
\begin{corollary}[SC-Adam]\label{suffcor}
Adam can converge if the following conditions are satisfied
\begin{enumerate}
\item $0\le\beta_k\le\beta<1$;
\item $0<\theta_k<1$;
\item There exists a non-increasing real sequence $(\alpha_k)_{k\ge 1}$ such that $\eta_k=\Theta(\alpha_{k})$;
\item $\left( \sum_{k=1}^{K}\eta_k{(1-\theta_k)}\right) \big/\left( K\eta_{K}\right) =o(1)$;
\item $\left( \sum_{k=1}^{K}\eta_k^2\right) \big/\left( K\eta_{K}\right) =o(1)$.\\
\end{enumerate}
\end{corollary}
	
\par The sufficient condition of \textbf{SC-Zou} for the convergence of Adam proposed by \citet[Corollary 9]{zou2019sufficient} is: 1.\ $0\le\beta_k\le\beta<1$; 2.\ $0<\theta_k<1$ and $\theta_k$ is non-decreasing; 3.\ There exist a non-increasing sequence $(\alpha_k)_{k\ge 1}$ and a positive constant $C$ such that $\alpha_k\le\eta_k/\sqrt{1-\theta_k}\le C\alpha_k$; 4.\ $\big(\sum_{k=1}^{K}\eta_k\sqrt{1-\theta_k}\big)\big/\left( K\eta_{K}\right) =o(1)$. It is worth mentioning  that \textbf{SC-Adam} in Corollary \ref{suffcor} is strictly weaker than \textbf{SC-Zou}. To illustrate the generality of our condition, we present some examples as follows.
\begin{itemize}
\item By choosing $\theta_k=1-1/k$ and $\eta_k=1/{k^q}$ for any $0<q<1/2$, it follows from the conclusion of Theorem \ref{th4} that Adam converges at a rate of $	O(1/K^{q})$. However, since $q<1/2$, then the sequence  $\eta_k/\sqrt{1-\theta_k}=k^{\frac{1}{2}-q}$ must be increasing. Therefore, such a choice of parameters does not satisfy the third point in \textbf{SC-Zou}.
\item By employing the selection  $\theta_k=1-\Theta(1/k)$ and $\eta_k=1/k^{q}$ for any $0<q< 1$, it can be deduced from Corollary \ref{suffcor} that Adam exhibits convergence. However, $\theta_k$ is not monotonically non-decreasing, thus failing to satisfy the second criterion in \textbf{SC-Zou}.
\end{itemize}
\par Moreover, Proposition \ref{contopro} in Appendix \ref{B} shows that \textbf{SC-Adam} is more relaxed than \textbf{SC-Zou} by rigorous proof.

\subsection{Almost Sure Convergence}\label{Almost Sure Convergence}
In this section, we provide the first almost sure convergence rate analysis for Adam with a decreasing step size.

\begin{theorem}\label{th6}
Let $(\mathbf{x}^k)_{k\ge 1}$ be the sequence generated by Algorithm \ref{Gadafom}, and  Assumptions \ref{Assunbia} and \ref{Assmin} hold. If we suppose that
\begin{align}\label{th6.2}
	\sum_{k=1}^{\infty}\frac{\eta_k}{\sum_{i=1}^{k-1}\eta_i}=\infty,\quad\sum_{k=1}^{\infty}\eta_{k}^2<\infty,\quad\sum_{k=1}^{\infty}\eta_{k}(1-\theta_k)<\infty,
\end{align}
and $\eta_k$ is decreasing, then for any $K\ge 1$, we have
\begin{align}
	\min_{1\le k\le K}\left\|  \nabla f(\mathbf{x}^k)\right\|  ^2=o\left( \frac{1}{\sum_{k=1}^{K}\eta_{k}}\right) \ a.s. \nonumber
\end{align}
\end{theorem}
 \begin{proof}
See Appendix \ref{Proof of th6}.
 \end{proof}

\par In the non-convex setting, \citet{lei2019stochastic} obtained a convergence rate of $O\big( {1}/{\sum_{k=1}^{K}\eta_{k}}\big)$ in expectation for SGD. In addition, 
\citet{sebbouh2021almost} derived the first almost sure convergence rate of $o\big( {1}/{\sum_{k=1}^{K}\eta_{k}}\big)$ for SGD when the objective functions are convex. Moreover, \citet{liu2022almost} analyzed the almost sure convergence of SGD, SHB, and SNAG, obtaining the rate of $o\big({1}/{\sum_{k=1}^{K}\eta_{k}}\big)$ for non-convex functions. To the best of our knowledge, Theorem \ref{th6} contributes the first almost sure convergence rate analysis for Adam in the non-convex setting, thereby significantly advancing the understanding of the trajectory-wise performance of Adam. Furthermore, by specifying the step sizes $\eta_{k}$, we attain a concrete convergence rate in the following corollary. 
	
\begin{corollary}\label{cor4}
In particular, if Assumptions \ref{Assunbia} and \ref{Assmin} are satisfied, then by choosing  $\eta_k=1/k^{q}$ for any $1/2<q< 1$ and $\theta_k=1-1/k^{p}$ for any $p>1-q$, we have
\begin{equation}
	\begin{aligned}
		\min_{1\le k\le K}\left\| \nabla f(\mathbf{x}^k)\right\|^2=o\left( \frac{1}{K^{1-q}}\right) \ a.s. \nonumber
	\end{aligned}
	\end{equation}
\end{corollary}
\begin{proof}
See Appendix \ref{Proof of cor4}.
\end{proof}
\par Note that by choosing $q\to 1/2$ in the step size $\eta_{k}$, then for $\theta_{k}=1-1/k^{p}$ with $p>1-q$, the almost sure convergence rate of Adam is arbitrarily close to  $o(1/\sqrt{K})$.

\section{Non-ergodic Convergence of Adam}\label{Non-ergodic Convergence of Adam}
In this section, we study the non-ergodic convergence of Adam in two aspects. Specifically, we establish the non-ergodic convergence for gradient sequences in Section \ref{Non-ergodic Convergence for Gradient Sequence}, followed by the analysis of the non-ergodic convergence for function values in Section \ref{PLErgcor}. In particular, as shown in Proposition \ref{prop-non}, the non-ergodic convergence developed in Theorems \ref{th2}, \ref{th5}, and \ref{th3} carries greater significance. 
\subsection{Non-ergodic Convergence for Gradient Sequences}\label{Non-ergodic Convergence for Gradient Sequence}	
In this section, we prove that the gradient norm of Adam with the last iterate output converges to zero.
\begin{theorem}\label{th2}
Let $(\mathbf{x}^k)_{k\ge 1}$ be the sequence generated by Algorithm \ref{Gadafom}, and Assumptions \ref{Assunbia} and \ref{Assmin} hold.  If we suppose   $\sum_{k=1}^{\infty}\eta_{k}=\infty$, $ \sum_{k=1}^{\infty}\eta_{k}^2<\infty$, $ \sum_{k=1}^{\infty}\eta_{k}(1-\theta_k)<\infty$ and there exists a non-increasing real sequence $(\alpha_k)_{k\ge 1}$ such that  $\eta_k=\Theta(\alpha_{k})$, then we have
\begin{equation}\label{th2.1}
	\begin{aligned}
		\lim_{k\rightarrow\infty} \left\| \nabla f(\mathbf{x}^k)\right\|=0\ a.s.,\ 
		\text{and}\ \lim_{k\rightarrow\infty} \mathbb{E}\left[\left\| \nabla f(\mathbf{x}^k)\right\|\right]=0.\nonumber
	\end{aligned}
\end{equation}
\end{theorem}
\begin{proof}
See Appendix \ref{Proof of th2}.
\end{proof}
\par In fact, the conditions of $\eta_{k}$ and $\theta_k$ in Theorem \ref{th2} are easy to be satisfied. On the one hand, the classical step size conditions  $\sum_{k=1}^{\infty}\eta_{k}=\infty$ and $\sum_{k=1}^{\infty}\eta_{k}^2<\infty$  are usually used in the analysis of stochastic algorithms  \citep[e.g.,][]{bertsekas2000gradient,bottou2018optimization,LIU202327}. In particular, \citet{new} argued that the conditions $\sum_{k=}^{\infty}\eta_{k}=\infty$ and $\sum_{k=}^{\infty}\eta_{k}^{2}<\infty$ are sufficient for the almost sure convergence of SGD in the strongly convex setting. On the other hand, the conditions  $\sum_{k=1}^{\infty}\eta_{k}(1-\theta_k)<\infty$ and $\eta_k=\Theta(\alpha_{k})$ with a non-increasing $\alpha_{k}$ often hold for adaptive step size algorithms. For example, \citet{chen2019convergence} obtained the convergence of AdaGrad with EMA momentum (AdaEMA) by choosing $\theta_k=1-1/k$ and $\eta_{k}=\eta/\sqrt{k}$, and \citet{mukkamala2017variants} achieved the convergence for RMSProp when setting $\theta_k=1-\alpha/k$ with $\alpha\in(0,1]$ and $\eta_{k}=\eta/\sqrt{k}$.

\subsection{Non-ergodic Convergence for Function Values}\label{PLErgcor}
In this section, we focus on the performance of Adam when the non-convex objectives satisfy the PL condition. Specifically, we obtain the non-ergodic convergence and non-ergodic convergence rate for function values in Theorems \ref{th5} and \ref{th3}, respectively.

\begin{assumption}\label{Asspl}	
The function $f$ satisfies the PL condition,  that is, for $\forall\,\mathbf{x}\in{\mathbb{R}^d}$, there exists $v>0$, such that $\left\| \nabla f(\mathbf{x})\right\|^2\ge2v(f(\mathbf{x})-f^*)$.
\end{assumption}

\par Note that if $f$ exhibits strong convexity, then $f$ must be a convex function that satisfies the PL condition. However, the converse is not necessarily valid, and \citet{karimi2016linear} provided the counterexample, which is $f(x)=x^{2}+3\sin^{2}x$.

\begin{theorem}\label{th5}
Suppose that the conditions in Theorem \ref{th2} still hold and Assumption  \ref{Asspl} is satisfied, then we have
\begin{equation}
	\begin{aligned}
		\lim_{k\rightarrow\infty} f(\mathbf{x}^k)=f^*\ a.s.,\ 
		\text{and}\ \lim_{k\rightarrow\infty} \mathbb{E}\left[f(\mathbf{x}^k)\right]=f^*.\nonumber
	\end{aligned}
\end{equation}
\end{theorem}
\begin{proof}
See Appendix \ref{Proof of th5}.
\end{proof}
	
\par Under the same step size condition as Theorem \ref{th2}, we achieve the convergence of function values for Adam with the last iterate output in Theorem \ref{th5}. Further, by choosing a specific step size $\eta_{k}$, we attain the non-ergodic convergence rate in the following theorem.

\begin{theorem}\label{th3}
Let $(\mathbf{x}^k)_{k\ge 1}$ be the sequence generated by Algorithm \ref{Gadafom}, and Assumptions \ref{Assunbia}, \ref{Assmin} and \ref{Asspl} are satisfied. If we choose $\eta_k=\frac{\sqrt{M^2+\epsilon}}{(1-\beta)v}\cdot\frac{1}{k+1}$, then for any $K\ge 1$, we have
\begin{equation}
	\begin{aligned}
		\mathbb{E}\left[ f(\mathbf{x}^{K})\right] -f^*
		\le
		&\left( \frac{C_4}{(1-\beta)^2}+C_6\right) \cdot\frac{ 1 }{K+1} +\frac{C_5}{(1-\beta)^2}\cdot\frac{\sum_{k=1}^{K}k(1-\theta_k)}{K(K+1)} ,\nonumber
	\end{aligned}
\end{equation}
where $C_4:=\frac{\beta LM^2(M^2+\epsilon)}{(1-\beta)^2 v^2\epsilon}$, $
C_5:=\frac{ M^4\sqrt{d(M^2+\epsilon)}}{(1-\beta)v\epsilon^{\frac{3}{2}}}$ and $
C_6:=\frac{LM^2(M^2+\epsilon)}{2(1-\beta)^2 v^2\epsilon}$ are all finite values.
\end{theorem}
\begin{proof}
See Appendix \ref{Proof of th3}.
\end{proof}

\par The PL condition is widely used in the convergence analysis of various stochastic optimization algorithms, including SGD  \citep{bassily2018exponential},  the norm-version of AdaGrad \citep{xie2020linear}, RMSProp \citep{sun2020novel}, and SSRGD \citep{li2022simple}. It is worth mentioning that this paper analyzes the performance of Adam under the PL condition for the first time.

\begin{corollary}\label{Nergco}
Suppose that the conditions in Theorem \ref{th3} still hold, then for $\theta_k=1-1/k^{p}$, we have
\begin{equation}
	\begin{aligned}
    	\mathbb{E}\left[ f(\mathbf{x}^{K+1})\right] -f^*
		\le
	\begin{cases}
	O({ 1 }/{K}) +O({1}/{K^p}), 
	& \text{$0<p< 2$}\\
	O({ 1 }/{K}) +O({\ln K}/{K^2}),
	& \text{$p=2$}\\
	O({ 1 }/{K}) +O({1}/{K^2}), 
	& \text{$2<p$}\\
	\end{cases},\nonumber
	\end{aligned}
\end{equation}
where $S_p$ is the constant such that $\sum_{k=1}^{K}(1/k^{p-1})=S_p<\infty$ for any $p>2$.
\end{corollary}
\par By choosing $\theta_{k}=1-1/k^{p}$, Corollary \ref{Nergco} achieves an $O(1/K)$ non-ergodic convergence rate of function values for Adam. Meanwhile, the convergence result improves as $p$ increases, which is consistent with previous work and the choice of hyperparameters in many practical applications. Specifically, \citet{zhang2022adam} declared that Adam can converge over a wide range of hyperparameter settings with increasing second-order momentum. Additionally, \citet{defossezsimple} showed that Adam exhibits convergence when $\theta_k\rightarrow 1$, while the counter-example  \citep[see][]{reddi2019convergence} employs $\theta_k=\theta<1/5$.

\section{Conclusion}\label{sec:conclusion}
In this paper, we have established a novel and comprehensive theoretical analysis for the Adam algorithm in the non-convex setting. Specifically, we have provided precise definitions of ergodic and non-ergodic convergence and analyzed the relationship between them, showing that non-ergodic convergence is much stronger and more realistic. Based on this, the convergence of Adam has been established in two cases. For the ergodic case, our analysis has obtained the expected and almost sure convergence rates of gradient sequences. For the non-ergodic case, we have achieved the convergence of the gradient norm and the convergence rate of function values. Additionally, this paper has proposed a more relaxed sufficient condition to guarantee convergence, which broadens the hyperparameter selection of Adam. In conclusion, this paper provides valuable insight into the convergence of Adam for non-convex problems and enhances the theoretical foundation for practical applications.

\section*{Acknowledgement}
This work was funded in part by the National Natural Science Foundation of China (Nos. 62176051, 62272096), in part by National Key R\&D Program of China (No. 2021YFA1003400), and in part by the Fundamental Research Funds for the Central Universities of China. (No. 2412020FZ024).

\newpage

\appendix
\section{Technical Lemmas}
\begin{lemma}\label{Lemsum}
Let $(b_k)_{k\ge 1}$ be a non-negative sequence of real values and  $0\le\beta<1$, then we have
\begin{align}
   	\label{Lemsum.1}
   	&\sum_{k=1}^{K}\sum_{j=1}^{k}\beta^{k-j}b_j\le\frac{1}{1-\beta}\sum_{k=1}^{K}b_k,\\
   	\label{Lemsum.2}
   	&\sum_{k=1}^{K}k\sum_{j=1}^{k}\beta^{k-j}b_j\le\frac{1}{(1-\beta)^2}\sum_{k=1}^{K}k b_k.
\end{align}
\end{lemma}
\begin{proof}
First, by exchanging the summation order, we then obtain
\begin{equation}
   	\begin{aligned}
   	&\sum_{k=1}^{K}\sum_{j=1}^{k}\beta^{k-j}b_j
   	=\sum_{j=1}^{K}\sum_{k=j}^{K}\beta^{k-j}b_j
   	=\sum_{j=1}^{K}\sum_{k=0}^{K-j}\beta^{k}b_j\\
   	&\le\sum_{j=1}^{K}\sum_{k=0}^{\infty}\beta^{k}b_j
   	\le\frac{1}{1-\beta}\sum_{j=1}^{K}b_j
   	=\frac{1}{1-\beta}\sum_{k=1}^{K}b_k.\nonumber
   	\end{aligned}
\end{equation}
This indicates that the conclusion in \eqref{Lemsum.1} holds.
\par Next, exchanging the order of summation for the left-hand side of \eqref{Lemsum.2} yields
\begin{equation}\label{equation}
   	\begin{aligned}
   	\sum_{k=1}^{K}k\sum_{j=1}^{k}\beta^{k-j}b_j
   	=\sum_{k=1}^{K}\sum_{j=1}^{k}k\beta^{k-j}b_j
   	=\sum_{j=1}^{K}\sum_{k=j}^{K}k\beta^{k-j}b_j
   	=\sum_{j=1}^{K}\beta^{-j}b_j\sum_{k=j}^{K}k\beta^{k}.
   	\end{aligned}
\end{equation}
Upon let $S_j^K:=\sum_{k=j}^{K}k\beta^{k}$, we have
\begin{align}
   	\label{Lemsum.3}&S_j^K=j \beta^j+(j+1) \beta^{j+1}+\cdots+K \beta^K,\\
   	\label{Lemsum.4}\beta &S_j^K=\ \ \ \ \ \ \ \ \ \ \ \ j \beta^{j+1}+\cdots+(K-1)\beta^K+K\beta^{K+1}.
\end{align}
Subtracting \eqref{Lemsum.4} from \eqref{Lemsum.3} gives
\begin{equation}\label{Lemsum.5}
   	\begin{aligned}
   	(1-\beta)S_j^K
   	&=j \beta^j+\beta^{j+1}+\cdots+ \beta^K-K\beta^{K+1}\\
   	&\overset{(a)}{\le}j \beta^j+\left( \beta^{j+1}+\cdots+ \beta^K\right) 
   	\\
   	&\overset{(b)}{\le}j \beta^j+\frac{\beta^{j+1}}{1-\beta},
   	\end{aligned}
\end{equation}
where $(a)$ is because $-K\beta^{K+1}<0$ and $(b)$ is due to $\beta<1$, that is, $\beta^{j+1}+\cdots+ \beta^K=\sum_{i=j+1}^{K}\beta^i\le\sum_{i=j+1}^{\infty}\beta^i=\frac{\beta^{j+1}}{1-\beta}$. Upon dividing both sides of \eqref{Lemsum.5} by $1-\beta$ gives
\begin{equation}
   	\begin{aligned}
   	S_j^K
   	=\sum_{k=j}^{K}k\beta^{k}
   	\le\frac{j \beta^j}{1-\beta}+\frac{\beta^{j+1}}{(1-\beta)^2}
   	\overset{(c)}{\le}\frac{j \beta^j}{1-\beta}+\frac{j\beta^{j+1}}{(1-\beta)^2}
   	=\frac{j\beta^{j}}{(1-\beta)^2},\nonumber
   	\end{aligned}
\end{equation}
where $(c)$ is because $j\ge 1$. Substituting the above equation back into \eqref{equation} yields
\begin{equation}
   	\begin{aligned}
   	\sum_{k=1}^{K}k\sum_{j=1}^{k}\beta^{k-j}b_j
   	=&\sum_{j=1}^{K}\beta^{-j}b_j\sum_{k=j}^{K}k\beta^{k}
   	\le\sum_{j=1}^{K}\beta^{-j}b_j\frac{j\beta^{j}}{(1-\beta)^2}\\
   	=&\frac{1}{(1-\beta)^2}\sum_{j=1}^{K}j b_j
   	=\frac{1}{(1-\beta)^2}\sum_{k=1}^{K}k b_k.\nonumber
   	\end{aligned}
\end{equation}
This completes the proof.
\end{proof}
\begin{lemma}\label{Lemmbound}
Let $(\mathbf{m}^k)_{k\ge 0}$ and $({\mathbf{v}}^k)_{k\ge 0}$ be the sequences generated  by the Algorithm \ref{Gadafom}, and assume that Assumption \ref{Assunbia} holds, then for $\forall\, k\ge 0$ we have
\begin{align}\label{Lemmbound.1}
	\lVert\mathbf{m}^k \rVert
	\le M\ a.s.,\quad 
	\lVert\mathbf{v}^k \rVert
	 \le M^2\ a.s.,\quad 
	\left\|  \frac{\mathbf{m}^k}{\sqrt{\mathbf{v}^k+\epsilon}} \right\|  \le \frac{M}{\sqrt{\epsilon}}\ a.s.
\end{align}
\end{lemma}
\begin{proof}
$(1)$ For the boundness of $\lVert\mathbf{m}^{k}\rVert$, we do mathematical induction on $k$. It follows from the initial value $\mathbf{m}^{0}=0$ that $\lVert\mathbf{m}^0 \rVert=0\le M$, then we assume that $\lVert\mathbf{m}^{k-1} \rVert\le M\ a.s.$ holds, thus
\begin{equation}
	\begin{aligned}
	\lVert \mathbf{m}^{k} \rVert &=\lVert \beta_k \mathbf{m}^{k-1} 
	+ (1-\beta_k)\mathbf{g}^k\rVert\\
	&\le\beta_k
	\lVert\mathbf{m}^{k-1}\rVert
	+(1-\beta_k)\lVert\mathbf{g}^k\rVert\\
	&\le \beta_k M+(1-\beta_k)M=M\ a.s.,\nonumber
	\end{aligned}
\end{equation}
which implies that $\parallel \mathbf{m}^k \parallel\le M\ a.s.$ holds for any $k\ge0$.
	
\par $(2)$ For the boundness of $\lVert \mathbf{v}^{k}\rVert$, we also do mathematical induction on $k$. Since $\lVert\mathbf{v}^0\rVert=0\le M^2$, we then assume that $\lVert \mathbf{v}^{k-1}\rVert\le M^2\ a.s.$, so
\begin{equation}
	\begin{aligned}
	\lVert \mathbf{v}^{k} \rVert
	&\,=\,\lVert  \theta_k\mathbf{v}^{k-1}+(1-\theta_k)(\mathbf{g}^k)^2\rVert
	\\
	&\le \theta_k\lVert \mathbf{v}^{k-1}\rVert
	+(1-\theta_k)\lVert (\mathbf{g}^k)^2 \rVert\\
	&\overset{(a)}{\le}\theta_k\lVert \mathbf{v}^{k-1}\rVert
	+(1-\theta_k)\lVert \mathbf{g}^k \rVert^2
	\\
	&\le\theta_kM^2+(1-\theta_k)M^2=M^2
	\ a.s.,\nonumber
	\end{aligned}
\end{equation}
where $(a)$ is because $\lVert \mathbf{x} ^2\rVert\le\left\|\mathbf{x}\right\|^2$ for $\forall\, \mathbf{x}\in\mathbb{R}^d$. Thus, for any $k\ge0$, we have $\lVert \mathbf{v}^k \rVert\le M^2\ a.s.$
	
\par $(3)$ By the definition of ${\ell}_2$ norm,  we have
\begin{equation}
	\begin{aligned}
	\left\| \frac{\mathbf{m}^k}{\sqrt{\mathbf{v}^k+\epsilon}} \right\|  &=\sqrt{\frac{(\mathbf{m}^k_1)^2}{\mathbf{v}^k_1+\epsilon}+\frac{(\mathbf{m}^k_2)^2}{\mathbf{v}^k_2+\epsilon}+\cdots+\frac{(\mathbf{m}^k_d)^2}{\mathbf{v}^k_d+\epsilon}}\\ &\overset{(b)}{\le}\sqrt{\frac{(\mathbf{m}^k_1)^2}{\epsilon}+\frac{(\mathbf{m}^k_2)^2}{\epsilon}+\cdots+\frac{(\mathbf{m}^k_d)^2}{\epsilon}}
	\\
	&=\frac{\parallel \mathbf{m}^k \parallel}{\sqrt{\epsilon}}\le\frac{M}{\sqrt{\epsilon}}\ a.s.,\nonumber
	\end{aligned}
\end{equation}
where $(b)$ is because  each component of $\mathbf{v}^k$ is non-negative (as shown in the fourth step of Algorithm \ref{Gadafom}). We have thus proved the lemma.
\end{proof}
\section{Proofs for Section \ref{ERGODIC}}
\subsection{Lemmas for Theorem \ref{th1}}
\begin{lemma}\label{LemGlow}
Let $(\mathbf{x}^k)_{k\ge 1}$ and $({\mathbf{v}}^k)_{k\ge 1}$ be the sequences generated by the Algorithm \ref{Gadafom}, and assume that Assumption \ref{Assunbia} holds, then for $\forall k\ge 1$ we have
\begin{align}\label{LemGlow.1}
   	\left\|  \frac{(\nabla f(\mathbf{x}^k))^2}{\sqrt{\mathbf{v}^{k-1}+\epsilon}} \right\|_1  \ge \frac{\left\| \nabla f(\mathbf{x}^k)\right\|^2}{\sqrt{M^2+\epsilon}}\ a.s.
   	\end{align}
\end{lemma}
   \begin{proof}
   	By the definition of $\ell_1$ norm, we have
   	\begin{equation}
   	\begin{aligned}
   	\left\|  \frac{(\nabla f(\mathbf{x}^k))^2}{\sqrt{\mathbf{v}^{k-1}+\epsilon}} \right\|_1
   	&=\frac{(\nabla_{1}  f(\mathbf{x}^k))^2}{\sqrt{\mathbf{v}^{k-1}_1+\epsilon}}
   	+\cdots+\frac{(\nabla_{d}  f(\mathbf{x}^k))^2}{\sqrt{\mathbf{v}^{k-1}_d+\epsilon}}
   	\ge\frac{(\nabla_{1}  f(\mathbf{x}^k))^2}{\sqrt{\left\| \mathbf{v}^{k-1}\right\| +\epsilon}}
   	+\cdots+\frac{(\nabla_{d}  f(\mathbf{x}^k))^2}{\sqrt{\left\| \mathbf{v}^{k-1}\right\| +\epsilon}}\\
   	&\overset{(a)}{\ge}\frac{(\nabla_{1}  f(\mathbf{x}^k))^2}{\sqrt{M^2+\epsilon}}
   	+\cdots+\frac{(\nabla_{d}  f(\mathbf{x}^k))^2}{\sqrt{M^2+\epsilon}}=\frac{\left\| \nabla f(\mathbf{x}^k)\right\|^2}{\sqrt{M^2+\epsilon}}\ a.s.,\nonumber
   	\end{aligned}
\end{equation}
where $(a)$ is due to the conclusion of Lemma \ref{Lemmbound},  that is,  $\lVert\mathbf{v}^k \rVert \le M^2\ a.s.$ This completes the proof.
\end{proof}
\begin{lemma}\label{yogiLem}
Let $({\mathbf{v}}^k)_{k\ge 1}$ be the sequence generated by the Algorithm \ref{Gadafom}, and assume that Assumption \ref{Assunbia} holds, then for $\forall k\ge 1$ we have
\begin{equation}\label{yogiLem.1}
   	\begin{aligned}
   	\left\| \frac{\mathbf{m}^k}{\sqrt{\mathbf{v}^{k-1}+\epsilon}}-\frac{\mathbf{m}^k}{\sqrt{\mathbf{v}^{k}+\epsilon}}\right\|
   	\le \frac{\sqrt{d}M^3}{\epsilon^{\frac{3}{2}}}(1-\theta_k)\ a.s.
   	\end{aligned}
\end{equation}
   	
\end{lemma}
   
\begin{proof}
Since $\left\| \mathbf{x}\right\|\le\left\| \mathbf{x}\right\|_1 $ holds for $\forall\,\mathbf{x}\in\mathbb{R}^d$, then we can get
\begin{equation}\label{yogi.1}
   	\begin{aligned}
   	&\left\| \frac{\mathbf{m}^k}{\sqrt{\mathbf{v}^{k-1}+\epsilon}}-\frac{\mathbf{m}^k}{\sqrt{\mathbf{v}^{k}+\epsilon}}\right\|
   	\le\left\| \frac{\mathbf{m}^k}{\sqrt{\mathbf{v}^{k-1}+\epsilon}}-\frac{\mathbf{m}^k}{\sqrt{\mathbf{v}^{k}+\epsilon}}\right\|_1\\
   	&=\sum_{i=1}^{d}\left| \frac{\mathbf{m}^k_i}{\sqrt{\mathbf{v}^{k-1}_i+\epsilon}}-\frac{\mathbf{m}^k_i}{\sqrt{\mathbf{v}^{k}_i+\epsilon}}\right| 
   	=\sum_{i=1}^{d}\left| \mathbf{m}^k_i \right| \left| \frac{1}{\sqrt{\mathbf{v}^{k-1}_i+\epsilon}}-\frac{1}{\sqrt{\mathbf{v}^{k}_i+\epsilon}}\right| \\
   	&\overset{(a)}{\le}
   	M\sum_{i=1}^{d}\left| \frac{1}{\sqrt{\mathbf{v}^{k-1}_i+\epsilon}}-\frac{1}{\sqrt{\mathbf{v}^{k}_i+\epsilon}}\right|
   	=M\sum_{i=1}^{d}\left|  \frac{\sqrt{\mathbf{v}^{k}_i+\epsilon}-\sqrt{\mathbf{v}^{k-1}_i+\epsilon}}{\sqrt{\mathbf{v}^{k-1}_i+\epsilon}\sqrt{\mathbf{v}^{k}_i+\epsilon}}\right| \\
   	&=M\sum_{i=1}^{d}\left|  \frac{\left( \sqrt{\mathbf{v}^{k}_i+\epsilon}-\sqrt{\mathbf{v}^{k-1}_i+\epsilon}\right) 
   	\left(  \sqrt{\mathbf{v}^{k}_i+\epsilon}+\sqrt{\mathbf{v}^{k-1}_i+\epsilon}\right) }{\sqrt{\mathbf{v}^{k-1}_i+\epsilon}\sqrt{\mathbf{v}^{k}_i+\epsilon}\left(  \sqrt{\mathbf{v}^{k}_i+\epsilon}+\sqrt{\mathbf{v}^{k-1}_i+\epsilon}\right)}\right| \\
   	&=M\sum_{i=1}^{d}\left|  \frac{\mathbf{v}^{k}_i-\mathbf{v}^{k-1}_i }{\sqrt{\mathbf{v}^{k-1}_i+\epsilon}\sqrt{\mathbf{v}^{k}_i+\epsilon}\left(  \sqrt{\mathbf{v}^{k}_i+\epsilon}+\sqrt{\mathbf{v}^{k-1}_i+\epsilon}\right)}\right| \\
   	&\overset{(b)}{\le}
   	\frac{M}{2\epsilon^{\frac{3}{2}}} \sum_{i=1}^{d}\left|  \mathbf{v}^{k}_i-\mathbf{v}^{k-1}_i \right| 
   	\overset{(c)}{=}\frac{M}{2\epsilon^{\frac{3}{2}}}(1-\theta_k) \sum_{i=1}^{d}\left| ({\mathbf{g}}^{k}_i)^2-{\mathbf{v}}^{k-1}_i\right| \ a.s.,
   	\end{aligned}
\end{equation}
where $(a)$ follows from the conclusion of Lemma \ref{Lemmbound}, that is,  $\lVert\mathbf{m}^k \rVert\le M\ a.s.$, $(b)$ is due to the fact that each component of $\mathbf{v}^k$ is non-negative, equation $(c)$ depends on the  iteration format  that $\mathbf{v}^k=\theta_k\mathbf{v}^{k-1}+(1-\theta_k)(\mathbf{g}^k)^2$, which means  ${\mathbf{v}}^{k}-{\mathbf{v}}^{k-1}=(1-\theta_k )\left( ({\mathbf{g}}^{k})^2-{\mathbf{v}}^{k-1}\right)$. Upon applying the absolute value inequality that $\left| a-b\right|\le\left| a\right|+\left| b\right| $ holds for $\forall\, a,\, b\in\mathbb{R}$, we have
\begin{equation}\label{yogi.2}
   	\begin{aligned}
   	&\sum_{i=1}^{d}\left| (\mathbf{g}^{k}_i)^2-\mathbf{v}^{k-1}_i\right|
   	\le \sum_{i=1}^{d}\left( \left| \mathbf{g}^{k}_i\right|^2 +\left| \mathbf{v}^{k-1}_i\right|\right)
    = \left\| {\mathbf{g}}^{k}\right\|^2 +\left\| {\mathbf{v}}^{k-1}\right\|_1
    \\
   	&\overset{(d)}{\le} M^2+\sqrt{d}\left\| {\mathbf{v}}^{k-1}\right\|
   	\overset{(e)}{\le} M^2+\sqrt{d}M^2
   	\le 2\sqrt{d}M^2 \ a.s.,
   	\end{aligned}
\end{equation}
where $(d)$ is because $\left\| \mathbf{x}\right\|_1\le\sqrt{d}\left\| \mathbf{x}\right\|$ holds for $\forall\,\mathbf{x}\in\mathbb{R}^d$, and the condition that $\left\| \mathbf{g}^k\right\| \le M\ a.s.$, $(e)$  follows from the conclusion $ 
\lVert \mathbf{v}^k \rVert\le M^2\ a.s.$ in Lemma \ref{Lemmbound}. 
Substituting \eqref{yogi.2} back into \eqref{yogi.1} gives
\begin{equation}
   	\begin{aligned}
   	\left\| \frac{\mathbf{m}^k}{\sqrt{\mathbf{v}^{k-1}+\epsilon}}-\frac{\mathbf{m}^k}{\sqrt{\mathbf{v}^{k}+\epsilon}}\right\|
   	&\le\frac{M}{2\epsilon^{\frac{3}{2}}}(1-\theta_k)2\sqrt{d}M^2
   	=\frac{\sqrt{d}M^3}{\epsilon^{\frac{3}{2}}}(1-\theta_k)\ a.s.\nonumber
   	\end{aligned}
\end{equation}
This completes the proof.
\end{proof}

\begin{lemma}\label{LemTheta}
Let $({\mathbf{x}}^k)_{k\ge 1}$, $({\mathbf{m}}^k)_{k\ge 1}$ and $({\mathbf{v}}^k)_{k\ge 1}$ be the sequences generated by the Algorithm \ref{Gadafom}. Suppose that Assumptions \ref{Assunbia} and \ref{Assmin}  hold, then for $\forall\, k\ge 1$, we have
\begin{align}\label{LemTheta.1}
   	\mathbb{E}\left[ \left\langle \nabla f(\mathbf{x}^k),\frac{\mathbf{m}^k}{{\sqrt{{\mathbf{v}}^k+\epsilon}}}\right\rangle\right]
   	\ge\sum_{j=1}^{k}\prod_{t=j+1}^{k}{\beta_{t}}D_j,\nonumber
\end{align}
where
\begin{align}
   	D_j:=&-L{\beta_j} \eta_{j-1}\mathbb{E}\left[\left\| \frac{ \mathbf{m}^{j-1}}{{\sqrt{\mathbf{v}^{j-1}+\epsilon}}}\right\|^2\right] 
   	+\frac{1-{\beta_j}}{{\sqrt{M^2+\epsilon}}}\mathbb{E}\left[\left\| \nabla f(\mathbf{x}^j)\right\|^2\right]
   	-\frac{d  M^4}{\epsilon^{\frac{3}{2}}}(1-\theta_j).\nonumber
\end{align}	
\end{lemma}
\begin{proof}
First, suppose that
\begin{align}
   	\Theta_k:=\mathbb{E}\left[ \left\langle \nabla f(\mathbf{x}^k),\frac{\mathbf{m}^k}{{\sqrt{{\mathbf{v}}^k+\epsilon}}}\right\rangle\right].\nonumber
\end{align}
then we can split $\Theta_k$ as follows
\begin{equation}\label{th1.7}
   	\begin{aligned}
   	\Theta_k=\underbrace{\mathbb{E}\left[ \left\langle \nabla f(\mathbf{x}^k),\frac{\mathbf{m}^k}{\sqrt{\mathbf{v}^{k-1}+\epsilon}}\right\rangle\right]}_{\Rmnum{1}}+\underbrace{\mathbb{E}\left[ \left\langle \nabla f(\mathbf{x}^k),\frac{\mathbf{m}^k}{\sqrt{{\mathbf{v}}^k+\epsilon}}-\frac{\mathbf{m}^k}{\sqrt{\mathbf{v}^{k-1}+\epsilon}}\right\rangle\right]}_{\Rmnum{2}}.
   	\end{aligned}
\end{equation}
For $\Rmnum{1}$, since  $\mathbf{m}^k=\beta_k\mathbf{m}^{k-1}+(1-\beta_k)\mathbf{g}^k$ in Algorithm \ref{Gadafom}, we then arrive at

\begin{equation}
	\begin{aligned}
		&\mathbb{E}\left[ \left\langle \nabla f(\mathbf{x}^k),\frac{\mathbf{m}^k}{\sqrt{\mathbf{v}^{k-1}+\epsilon}}\right\rangle \bigg| \mathcal{F}^k\right]
		=\mathbb{E}\left[ \left\langle \nabla f(\mathbf{x}^k),\frac{\beta_k \mathbf{m}^{k-1}+(1-\beta_k )\mathbf{g}^k}{\sqrt{\mathbf{v}^{k-1}+\epsilon}}\right\rangle \bigg| \mathcal{F}^k\right]\\
		&{\overset{(a)}{=}}\beta_k\left\langle \nabla f(\mathbf{x}^k),\frac{ \mathbf{m}^{k-1}}{\sqrt{\mathbf{v}^{k-1}+\epsilon}}\right\rangle
		+(1-\beta_k)\left\langle \nabla f(\mathbf{x}^k),\frac{\mathbb{E}\left[ \mathbf{g}^k|
		\mathcal{F}^k\right] }{\sqrt{\mathbf{v}^{k-1}+\epsilon}}\right\rangle\\
		&{\overset{(b)}{=}}
		\beta_k\left\langle \nabla f(\mathbf{x}^k),\frac{ \mathbf{m}^{k-1}}{\sqrt{\mathbf{v}^{k-1}+\epsilon}}\right\rangle+(1-\beta_k)\left\langle \nabla f(\mathbf{x}^k),\frac{\nabla f(\mathbf{x}^k) }{\sqrt{\mathbf{v}^{k-1}+\epsilon}}\right\rangle,\nonumber
	\end{aligned}
\end{equation}
where $(a)$ is because $\mathbf{x}^{k}$, $\mathbf{m}^{k-1}$ and $\mathbf{v}^{k-1}$ are $\mathcal{F}^k$-measurable, $(b)$ is because of the unbiased estimation assumption,  that is,   $\mathbb{E}[ \mathbf{g}^k|\mathcal{F}^k] = \nabla f(\mathbf{x}^k)$. Further, from the fact that $\langle\nabla f(\mathbf{x}^{k}),\nabla f(\mathbf{x}^{k})/\sqrt{\mathbf{v}^{k-1}+\epsilon}\rangle=\lVert\nabla \big(f(\mathbf{x}^{k})\big)^{2}/\sqrt{\mathbf{v}^{k-1}+\epsilon}\rVert_{1}$, we have
\begin{equation}\label{th1.3}
	\begin{aligned}
		\mathbb{E}\left[ \left\langle \nabla f(\mathbf{x}^k),\frac{\mathbf{m}^k}{\sqrt{\mathbf{v}^{k-1}+\epsilon}}\right\rangle \bigg| \mathcal{F}^k\right]
		&{=}\beta_k\left\langle \nabla f(\mathbf{x}^k),\frac{ \mathbf{m}^{k-1}}{\sqrt{\mathbf{v}^{k-1}+\epsilon}}\right\rangle+(1-\beta_k)\left\| \frac{(\nabla f(\mathbf{x}^k))^2 }{\sqrt{\mathbf{v}^{k-1}+\epsilon}}\right\|_1\\
		&{\overset{(c)}{=}}
		\beta_k\left\langle \nabla f(\mathbf{x}^{k-1}),\frac{ \mathbf{m}^{k-1}}{\sqrt{\mathbf{v}^{k-1}+\epsilon}}\right\rangle
		+(1-\beta_k)\left\| \frac{(\nabla f(\mathbf{x}^k))^2 }{\sqrt{\mathbf{v}^{k-1}+\epsilon}}\right\|_1
		\\
		&{-}\beta_k\left\langle \nabla f(\mathbf{x}^{k-1})-\nabla f(\mathbf{x}^k),\frac{ \mathbf{m}^{k-1}}{\sqrt{\mathbf{v}^{k-1}+\epsilon}}\right\rangle
		,
	\end{aligned}
\end{equation}
where $(c)$ is obtained by  adding and subtracting term  $\beta_k\langle  \nabla f(\mathbf{x}^{k-1}), \mathbf{m}^{k-1}/\sqrt{\mathbf{v}^{k-1}+\epsilon}\rangle $.

\par For the third term on the right-hand side of \eqref{th1.3},  we have
\begin{equation}\label{th1.4}
   	\begin{aligned}
   	&-\beta_k\left\langle \nabla f(\mathbf{x}^{k-1})-\nabla f(\mathbf{x}^k),\frac{ \mathbf{m}^{k-1}}{\sqrt{\mathbf{v}^{k-1}+\epsilon}}\right\rangle
   	\overset{(d)}{\ge}-\beta_k\left\|  \nabla f(\mathbf{x}^{k-1})-\nabla f(\mathbf{x}^k)\right\| \left\| \frac{ \mathbf{m}^{k-1}}{\sqrt{\mathbf{v}^{k-1}+\epsilon}}\right\| \\
   	&\overset{(e)}{\ge}
   	-\beta_k L\left\|   \mathbf{x}^{k-1}-\mathbf{x}^k\right\| \left\| \frac{ \mathbf{m}^{k-1}}
   	{\sqrt{\mathbf{v}^{k-1}+\epsilon}}
   	\right\| 
   	\overset{(f)}{=}-L\beta_k \eta_{k-1}\left\| \frac{ \mathbf{m}^{k-1}}{\sqrt{\mathbf{v}^{k-1}+\epsilon}}\right\|^2,
   	\end{aligned}
\end{equation}
where $(d)$ uses the Cauchy-Schwarz inequality $\left\langle a,b\right \rangle \le\left\| a\right\| \left\| b\right\| $, $(e)$ is due to Assumption \ref {Assmin} that $f$ is $L$-smooth, $(f)$ is because of the iterate format   $\mathbf{x}^{k+1}=\mathbf{x}^k-\eta_k{\mathbf{m}^k}/{\sqrt{{\mathbf{v}}^k+\epsilon}}$. Upon substituting \eqref{th1.4} back into \eqref{th1.3}, we have
\begin{equation}\label{en.1}
   	\begin{aligned}
   	\mathbb{E}\left[ \left\langle \nabla f(\mathbf{x}^k),\frac{\mathbf{m}^k}{\sqrt{\mathbf{v}^{k-1}+\epsilon}}\right\rangle \bigg| \mathcal{F}^k\right]
   	&\ge \beta_k\left\langle \nabla f(\mathbf{x}^{k-1}),\frac{ \mathbf{m}^{k-1}}{\sqrt{\mathbf{v}^{k-1}+\epsilon}}\right\rangle\\
   	&-L\beta_k \eta_{k-1}\left\| \frac{ \mathbf{m}^{k-1}}{\sqrt{\mathbf{v}^{k-1}+\epsilon}}\right\|^2
   	+(1-\beta_k)\left\| \frac{(\nabla f(\mathbf{x}^k))^2 }{\sqrt{\mathbf{v}^{k-1}+\epsilon}}\right\|_1.
   	\end{aligned}
\end{equation} 
Taking total expectation on both sides of \eqref{en.1}, since the tower property of the conditional expectation gives $\mathbb{E}[\mathbb{E}[\,\cdot\,|\mathcal{F}^k]]=\mathbb{E}[\,\cdot\,]$, then we can get
\begin{equation}\label{th1.5}
   	\begin{aligned}
   	\uppercase\expandafter
   	{\romannumeral1}
   	&=\mathbb{E}\left[ \left\langle \nabla f(\mathbf{x}^k),\frac{\mathbf{m}^k}{\sqrt{\mathbf{v}^{k-1}+\epsilon}}\right\rangle\right]
   	=\mathbb{E}\left[\mathbb{E}\left[ \left\langle \nabla f(\mathbf{x}^k),\frac{\mathbf{m}^k}{\sqrt{\mathbf{v}^{k-1}+\epsilon}}\right\rangle\bigg|\mathcal{F}^k\right]\right]\\
   	&\ge\beta_k\mathbb{E}\left[\left\langle \nabla f(\mathbf{x}^{k-1}),\frac{ \mathbf{m}^{k-1}}{\sqrt{\mathbf{v}^{k-1}+\epsilon}}\right\rangle\right]
   	\\
   	&
   	-L\beta_k \eta_{k-1}\mathbb{E}\left[\left\| \frac{ \mathbf{m}^{k-1}}{\sqrt{\mathbf{v}^{k-1}+\epsilon}}\right\|^2\right] 
   	+(1-\beta_k)\mathbb{E}\left[\left\| \frac{(\nabla f(\mathbf{x}^k))^2 }{\sqrt{\mathbf{v}^{k-1}+\epsilon}}\right\|_1\right] \\
   	&\overset{(g)}{\ge}
   	\beta_k\Theta_{k-1}
   	-L\beta_k \eta_{k-1}\mathbb{E}\left[\left\| \frac{ \mathbf{m}^{k-1}}{\sqrt{\mathbf{v}^{k-1}+\epsilon}}\right\|^2\right] 
   	+\frac{1-\beta_k}{\sqrt{M^2+\epsilon}}\mathbb{E}\left[\left\| \nabla f(\mathbf{x}^k)\right\|^2\right] ,
   	\end{aligned}
\end{equation}
where $(g)$ applies Lemma \ref{LemGlow},  that is, 
$\|{(\nabla f(\mathbf{x}^k))^2}/{\sqrt{\mathbf{v}^{k-1}+\epsilon}}\|_1  \ge {\left\|  \nabla f(\mathbf{x}^k)\right\|^2}/{\sqrt{M^2+\epsilon}}\ a.s.$\\
   	
On the other hand, for $\Rmnum{2}$, it is easy to obtain that
\begin{equation}\label{th1.6}
   	\begin{aligned}
   	\Rmnum{2}
   	&{=}\mathbb{E}
   	\left[ \left\langle \nabla f(\mathbf{x}^k),\frac{\mathbf{m}^k}{\sqrt{{\mathbf{v}}^k+\epsilon}}-\frac{\mathbf{m}^k}{\sqrt{\mathbf{v}^{k-1}+\epsilon}}\right\rangle\right]
   	\\
   	&{=}-\mathbb{E}\left[ \left\langle \nabla f(\mathbf{x}^k),\frac{\mathbf{m}^k}{\sqrt{\mathbf{v}^{k-1}+\epsilon}}-\frac{\mathbf{m}^{k}}{\sqrt{\mathbf{v}^{k}+\epsilon}}\right\rangle\right]\\
   	&{\overset{(h)}{\ge}}
   	-\mathbb{E}\left[\left\|  \nabla f(\mathbf{x}^k)\right\|  \left\| \frac{\mathbf{m}^k}{\sqrt{\mathbf{v}^{k-1}+\epsilon}}-\frac{\mathbf{m}^k}{\sqrt{\mathbf{v}^{k}+\epsilon}}\right\| \right]\\
   	&{\overset{(i)}{\ge}}
   	-\frac{\sqrt{d} M^4}{\epsilon^{\frac{3}{2}}}(1-\theta_k),
   	\end{aligned}
\end{equation}
where $(h)$ uses the Cauchy-Schwarz inequality $\left\langle a,b\right\rangle \le\left\| a\right\| \left\| b\right\| $, $(i)$ relies on  Lemma \ref{yogiLem},    that is,  $\| {\mathbf{m}^k}/{\sqrt{\mathbf{v}^{k-1}+\epsilon}}-{\mathbf{m}^k}/{\sqrt{\mathbf{v}^{k}+\epsilon}}\|\le {\sqrt{d}M^3}(1-\theta_k)/{\epsilon^{\frac{3}{2}}}\ a.s.$ and Assumption \ref{Assunbia},  that is,   $\|  \nabla f(\mathbf{x}^k)\|  =\| \mathbb{E}\left[ \mathbf{g}^k|\mathcal{F}^k\right] \|  \le\mathbb{E}\left[ \left\| \mathbf{g}^k\right\|  |\mathcal{F}^k\right]\le M\ a.s.$ Upon plugging  \eqref{th1.5} and  \eqref{th1.6} back into  \eqref{th1.7}, we obtain
\begin{equation}
   	\begin{aligned}
   	\Theta_k
   	&\ge\beta_k\Theta_{k-1}
   	\\
   	&-L\beta_k \eta_{k-1}\mathbb{E}\left[\left\| \frac{ \mathbf{m}^{k-1}}{\sqrt{\mathbf{v}^{k-1}+\epsilon}}\right\|^2\right] 
   	+\frac{1-\beta_k}{\sqrt{M^2+\epsilon}}\mathbb{E}\left[\left\| \nabla f(\mathbf{x}^k)\right\|^2\right]-\frac{\sqrt{d} M^4}{\epsilon^{\frac{3}{2}}}(1-\theta_k).
   	\end{aligned}
\end{equation}
Upon let
\begin{equation}\label{th1.15}
   	\begin{aligned}
   	D_k:=&-L\beta_k \eta_{k-1}\mathbb{E}\left[\left\| \frac{ \mathbf{m}^{k-1}}{\sqrt{\mathbf{v}^{k-1}+\epsilon}}\right\|^2\right] 
   	+\frac{1-\beta_k}{\sqrt{M^2+\epsilon}}\mathbb{E}\left[\left\| \nabla f(\mathbf{x}^k)\right\|^2\right]
   	-\frac{\sqrt{d} M^4}{\epsilon^{\frac{3}{2}}}(1-\theta_k).\nonumber
   	\end{aligned}
\end{equation}
Then we can conclude that
\begin{equation}\label{th1.8}
   	\begin{aligned}
   	\Theta_k
   	&\ge\beta_k\Theta_{k-1}+D_k
   	\ge\beta_k\beta_{k-1}\Theta_{k-2}+\beta_k D_{k-1}+D_k\\
    &\ge\,\cdots\,
   	\ge\prod_{j=1}^{k}\beta_{j}\Theta_0+\sum_{j=1}^{k}\prod_{t=j+1}^{k}\beta_{t}D_j
   	\overset{(j)}{=}\sum_{j=1}^{k}\prod_{t=j+1}^{k}\beta_{t}D_j,
   	\end{aligned}
\end{equation}
where $(j)$ uses $\Theta_0=\mathbb{E}\big[\langle \nabla f(\mathbf{x}^{0}),{ \mathbf{m}^{0}}/{\sqrt{{\mathbf{v}}^{0}+\epsilon}}\rangle\big]
=\langle \nabla f(\mathbf{x}^{0}),{ \mathbf{m}^{0}}/{\sqrt{{\mathbf{v}}^{0}+\epsilon}}\rangle=0$. Here, we follow the convention that $\prod_{t=k+1}^{k}\beta_{t}=1$.The proof of the lemma is now complete.
\end{proof}
\subsection{Proof of Theorem \ref{th1}}\label{Proof of th1}
Combining the condition that $f$ is $L$-smooth and the Descent Lemma in  \citep{nesterov2003introductory}, we have
\begin{equation}
   	\begin{aligned}
   	f(\mathbf{x}^{k+1})\le&f(\mathbf{x}^k)+\left\langle \nabla f(\mathbf{x}^k),\mathbf{x}^{k+1}-\mathbf{x}^k\right\rangle +\frac{L}{2}\left\|  \mathbf{x}^{k+1}-\mathbf{x}^k\right\|  ^2\\
   	\overset{(a)}{=}&f(\mathbf{x}^k)-\eta_k\left\langle \nabla f(\mathbf{x}^k),\frac{\mathbf{m}^k}{\sqrt{{\mathbf{v}}^k+\epsilon}}\right\rangle +\frac{L\eta_k^2}{2}\left\|  \frac{\mathbf{m}^k}{\sqrt{{\mathbf{v}}^k+\epsilon}}\right\| ^2,\nonumber
   	\end{aligned}
\end{equation}
where $(a)$ is due to the iteration format $\mathbf{x}^{k+1}=\mathbf{x}^k-\eta_k{\mathbf{m}^k}/{\sqrt{{\mathbf{v}}^k+\epsilon}}$. Upon taking the  total expectation on both sides of the above equation gives
\begin{equation}\label{th1.2}
   	\begin{aligned}
   	\mathbb{E}\left[ f(\mathbf{x}^{k+1})\right]
   	\le&\mathbb{E}\left[ f(\mathbf{x}^k)\right] -\eta_k\mathbb{E}\left[ \left\langle \nabla f(\mathbf{x}^k),\frac{\mathbf{m}^k}{\sqrt{{\mathbf{v}}^k+\epsilon}}\right\rangle\right] +\frac{L\eta_k^2}{2}\mathbb{E}\left[ \left\|  \frac{\mathbf{m}^k}{\sqrt{{\mathbf{v}}^k+\epsilon}}\right\| ^2\right] \\
   	\overset{(b)}{\le}&\mathbb{E}\left[ f(\mathbf{x}^k)\right] -\eta_k\sum_{j=1}^{k}\prod_{t=j+1}^{k}\beta_{t}D_j  +\frac{L\eta_k^2}{2}\mathbb{E}\left[ \left\|  \frac{\mathbf{m}^k}{\sqrt{{\mathbf{v}}^k+\epsilon}}\right\| ^2\right] \\
   	=&\mathbb{E}\left[ f(\mathbf{x}^k)\right] -\eta_k\sum_{j=1}^{k}\prod_{t=j+1}^{k}\beta_{t}\left( -L\beta_{j} \eta_{j-1}\mathbb{E}\left[\left\| \frac{ \mathbf{m}^{j-1}}{\sqrt{\mathbf{v}^{j-1}+\epsilon}}\right\|^2\right] 
   	\right.\\
   	+&\left.\frac{1-\beta_{j}}{\sqrt{M^2+\epsilon}}\mathbb{E}\left[\left\| \nabla f(\mathbf{x}^j)\right\|^2\right]
   	-\frac{\sqrt{d}  M^4}
   	{\epsilon^{\frac{3}{2}}}(1-\theta_j)\right) +\frac{L\eta_k^2}{2}\mathbb{E}\left[ \left\|  \frac{\mathbf{m}^k}{\sqrt{{\mathbf{v}}^k+\epsilon}}\right\| ^2\right], \nonumber
   	\end{aligned}
\end{equation}
where $(b)$ follows from Lemma \ref{LemTheta}. Further calculation of the above equation yields
\begin{equation}\label{th1.22}
   	\begin{aligned}
   	\mathbb{E}\left[ f(\mathbf{x}^{k+1})\right]
   	&\le\mathbb{E}\left[ f(\mathbf{x}^k)\right] 
   	-\frac{1}{\sqrt{M^2+\epsilon}}\eta_k\sum_{j=1}^{k}\prod_{t=j+1}^{k}\beta_{t}(1-\beta_{j})\mathbb{E}\left[\left\| \nabla f(\mathbf{x}^j)\right\|^2\right]\\
   	&+\frac{L\eta_k^2}{2}\mathbb{E}\left[ \left\|  \frac{\mathbf{m}^k}{\sqrt{{\mathbf{v}}^k+\epsilon}}\right\| ^2\right]
   	+L\eta_k\sum_{j=1}^{k}\prod_{t=j}^{k}\beta_{t} \eta_{j-1}\mathbb{E}\left[\left\| \frac{ \mathbf{m}^{j-1}}{\sqrt{\mathbf{v}^{j-1}+\epsilon}}\right\|^2\right]\\
   	&+\frac{\sqrt{d}  M^4}{\epsilon^{\frac{3}{2}}}\eta_k\sum_{j=1}^{k}\prod_{t=j+1}^{k}\beta_{t}(1-\theta_j).\nonumber
   	\end{aligned}
\end{equation}
Upon by the condition that $0\le\beta_k\le\beta<1$, we can obtain
\begin{equation}\label{th1.17}
   	\begin{aligned}
   	\mathbb{E}\left[ f(\mathbf{x}^{k+1})\right]
   	&\overset{(c)}{\le}
   	\mathbb{E}\left[ f(\mathbf{x}^k)\right] 
   	-\frac{1-\beta}{\sqrt{M^2+\epsilon}}\eta_k\mathbb{E}\left[\left\| \nabla f(\mathbf{x}^k)\right\|^2\right]
   	+\frac{L\eta_k^2}{2}\mathbb{E}\left[ \left\|  \frac{\mathbf{m}^k}{\sqrt{{\mathbf{v}}^k+\epsilon}}\right\| ^2\right]\\
   	&+\beta L\eta_k\sum_{j=1}^{k}\beta^{k-j} \eta_{j-1}\mathbb{E}\left[\left\| \frac{ \mathbf{m}^{j-1}}{\sqrt{\mathbf{v}^{j-1}+\epsilon}}\right\|^2\right]
   	+\frac{\sqrt{d} M^4}{\epsilon^{\frac{3}{2}}}\eta_k\sum_{j=1}^{k}\beta^{k-j}(1-\theta_j) ,
   	\end{aligned}
\end{equation}
where $(c)$ is due to $\prod_{t=j}^{k}\beta_{t}\le\beta^{k-j+1}$ and $\prod_{t=j+1}^{k}\beta_{t}\le\beta^{k-j}$ when $0\le\beta_k\le\beta<1$, and the fact that
\begin{equation}
   	\begin{aligned}
   	&\sum_{j=1}^{k}\prod_{t=j+1}^{k}\beta_{t}(1-\beta_{j})\mathbb{E}\left[\left\| \nabla f(\mathbf{x}^j)\right\|^2\right]
   	\overset{(d)}{\ge}\prod_{t=k+1}^{k}\beta_{t}(1-\beta_k)\mathbb{E}\left[\left\| \nabla f(\mathbf{x}^k)\right\|^2\right]\\
    &	\overset{(e)}{=}(1-\beta_k)\mathbb{E}\left[\left\| \nabla f(\mathbf{x}^k)\right\|^2\right]
   	\ge(1-\beta)\mathbb{E}\left[\left\| \nabla f(\mathbf{x}^k)\right\|^2\right],\nonumber
   	\end{aligned}
\end{equation}
where $(d)$ is because $\prod_{t=j+1}^{k}\beta_{t}(1-\beta_{j})\mathbb{E}\big[\| \nabla f(\mathbf{x}^j)\|^2\big]\ge 0$ holds for any $j$ and $\sum_{j=1}^{k}a_j\ge a_k$ holds for $a_j\ge 0$ with $a_j:=\prod_{t=j+1}^{k}\beta_{t}(1-\beta_{j})\mathbb{E}\big[\| \nabla f(\mathbf{x}^j)\|^2\big]$, $(e)$ holds by the stipulation that $\prod_{t=k+1}^{k}\beta_{t}=1$. Recalling
that $\|  {\mathbf{m}^k}/{\sqrt{{\mathbf{v}}^k+\epsilon}}\|^2\le {M^2}/{\epsilon}\ a.s.$ is obtained in Lemma \ref{Lemmbound}, therefore, \eqref{th1.17} can be estimated as
\begin{equation}\label{th1.14}
   	\begin{aligned}
   	\mathbb{E}\left[ f(\mathbf{x}^{k+1})\right] 
   	&\le
   	\mathbb{E}\left[ f(\mathbf{x}^k)\right] 
   	-\frac{1-\beta}{\sqrt{M^2+\epsilon}}\eta_k\mathbb{E}\left[\left\| \nabla f(\mathbf{x}^k)\right\|^2\right]\\
   	&+\frac{\beta L M^2}{\epsilon}\eta_k\sum_{j=1}^{k}\beta^{k-j} \eta_{j-1}
   	+\frac{\sqrt{d} M^4}{\epsilon^{\frac{3}{2}}}\eta_k\sum_{j=1}^{k}\beta^{k-j}(1-\theta_j)
   	+\frac{M^2L}{2\epsilon}\eta_k^2 .
   	\end{aligned}
\end{equation}
Upon rearranging the above equation and summing both sides over $k$ from $1$ to $K$, we have
\begin{equation}\label{th1.12}
   	\begin{aligned}
   	&\frac{1-\beta}{\sqrt{M^2+\epsilon}}\sum_{k=1}^{K}\eta_k\mathbb{E}\left[\left\| \nabla f(\mathbf{x}^k)\right\|^2\right]\\
   	&\le f(\mathbf{x}^1)-f^*
   	+\frac{\beta L M^2}{\epsilon}\sum_{k=1}^{K}\eta_k\sum_{j=1}^{k}\beta^{k-j} \eta_{j-1}
   	+\frac{\sqrt{d} M^4}{\epsilon^{\frac{3}{2}}}\sum_{k=1}^{K}\eta_k\sum_{j=1}^{k}\beta^{k-j}(1-\theta_j)
   	+\frac{M^2L}{2\epsilon}\sum_{k=1}^{K}\eta_k^2.
   	\end{aligned}
\end{equation}
By the condition that $f(\mathbf{x})\geq f^{*}$ in Assumption \ref{Assmin}, the first two terms in the right-hand side of \eqref{th1.12} are derived from $\sum_{k=1}^{K}\left( \mathbb{E}[ f(\mathbf{x}^{k})]-\mathbb{E}[ f(\mathbf{x}^{k+1})]\right) =f(\mathbf{x}^{1})-\mathbb{E}[ f(\mathbf{x}^{K+1})] \le f(\mathbf{x}^{1})-f^*$. 
\par Next, we estimate the second term on the right-hand side of \eqref{th1.12} as follows
\begin{equation}\label{th1.10}
   	\begin{aligned}
   	&\sum_{k=1}^{K}\eta_k\sum_{j=1}^{k}{\beta}^{k-j} \eta_{j-1}
   	\overset{(f)}{\le}\sum_{k=1}^{K}\tilde{C}_0 \alpha_k\sum_{j=1}^{k}{\beta}^{k-j} \tilde{C}_0 \alpha_{j-1}
   	=\tilde{C}_0^2\sum_{k=1}^{K}\alpha_k\sum_{j=1}^{k}{\beta}^{k-j} \alpha_{j-1}\\
   	&\overset{(g)}{\le}
   	\tilde{C}_0^2 \sum_{k=1}^{K}\sum_{j=1}^{k}{\beta}^{k-j} \alpha_j \alpha_{j-1}
   	\overset{(h)}{\le}\frac{\tilde{C}_0^2}{1-{\beta}}\sum_{k=1}^{K}\alpha_k \alpha_{k-1}
   	\overset{(i)}{=}\frac{\tilde{C}_0^2}{1-{\beta}}\sum_{k=2}^{K}\alpha_k \alpha_{k-1}\\
   	&\overset{(j)}{\le}
   	\frac{\tilde{C}_0^2}{1-{\beta}}\sum_{k=2}^{K}\alpha_{k-1}^2
   	=\frac{\tilde{C}_0^2}{1-{\beta}}\sum_{k=1}^{K-1}\alpha_{k}^2 
   	\le\frac{\tilde{C}_0^2}{1-{\beta}} \sum_{k=1}^{K} \alpha_{k}^2
   	\overset{(k)}{\le}\frac{\tilde{C}_0^2}{C_0^2(1-{\beta})}\sum_{k=1}^{K} \eta_{k}^2.
   	\end{aligned}
\end{equation}
Without loss of generality, we assume that $\eta_0=0$ in \eqref{th1.8}. Thus, we can deduce that $\alpha_0=0$, leading to the fulfilment of equation 
$(i)$. Since the inequality $C_0\alpha_{k}\le\eta_{k}\le \tilde{C}_0 \alpha_k$ holds for $\forall k\ge 0$, which makes  $(f)$ and $(k)$ hold. $(g)$ and $(j)$ can be inferred by using the fact that $\alpha_k$ is non-increasing. $(h)$ uses \eqref{Lemsum.1} in Lemma \ref{Lemsum},  that is,  $\sum_{k=1}^{K}\sum_{j=1}^{k}{\beta}^{k-j}b_j\le\frac{1}{1-{\beta}}\sum_{k=1}^{K}b_k$ with $b_j:=\alpha_j \alpha_{j-1}$. Upon for  the third term of the right-hand side of \eqref{th1.12}, we have
\begin{equation}\label{th1.11}
   	\begin{aligned}
   	&\sum_{k=1}^{K}\eta_k\sum_{j=1}^{k}{\beta}^{k-j}{(1-\theta_j)}
   	\overset{(l)}{\le}\tilde{C}_0\sum_{k=1}^{K}\sum_{j=1}^{k}{\beta}^{k-j}\alpha_j{(1-\theta_j)}\\
   	&\overset{(m)}{\le}
   	\frac{\tilde{C}_0 }{1-\beta}\sum_{k=1}^{K}\alpha_k{(1-\theta_k)}
   	\overset{(n)}{\le}\frac{\tilde{C}_0 }{C_0(1-{\beta})}\sum_{k=1}^{K}\eta_k{(1-\theta_k)},
   	\end{aligned}
\end{equation}
where $(l)$ and $(n)$ use the conditions that $\eta_{k}=\Theta(\alpha_{k})$ and $\alpha_k$ is non-increasing, $(m)$ applies \eqref{Lemsum.1} in Lemma \ref{Lemsum},  that is,   $\sum_{k=1}^{K}\sum_{j=1}^{k}{\beta}^{k-j}b_j\le\frac{1}{1-{\beta}}\sum_{k=1}^{K}b_k$, in which we let $b_j:=\alpha_j{(1-\theta_j)}$.
Now, plugging \eqref{th1.10} and \eqref{th1.11} back into \eqref{th1.12}, we then obtain
\begin{equation}\label{en.7}
   	\begin{aligned}
   	&\frac{1-\beta}{\sqrt{M^2+\epsilon}}\sum_{k=1}^{K}\eta_k\mathbb{E}\left[\left\| \nabla f(\mathbf{x}^k)\right\|^2\right]\\
   	&\le f(\mathbf{x}^1)-f^*
   	+\frac{\tilde{C}_0^2\beta L M^2}{\epsilon C_0^2(1-\beta)} \sum_{k=1}^{K}\eta_{k}^2 
   	+{\frac{\tilde{C}_0\sqrt{d}  M^4}{\epsilon^{\frac{3}{2}}C_0(1-\beta)}}\sum_{k=1}^{K}\eta_k(1-\theta_k)
   	+\frac{M^2L}{2\epsilon}\sum_{k=1}^{K}\eta_k^2.
   	\end{aligned}
\end{equation}
Multiplying both sides of the above equation by $\frac{\sqrt{M^2+\epsilon}}{1-\beta}$ gives
\begin{equation}\label{th1.20}
   	\begin{aligned}
   	\sum_{k=1}^{K}\eta_k\mathbb{E}\left[\left\| \nabla f(\mathbf{x}^k)\right\|^2\right]
   	&\le\frac{\sqrt{M^2+\epsilon}\left( f(\mathbf{x}^1)-f^*\right) }{1-\beta}
   	+\frac{\tilde{C}_0 M^4\sqrt{d(M^2+\epsilon)}}{\epsilon^{\frac{3}{2}}C_0(1-\beta)^2}\sum_{k=1}^{K}\eta_k(1-\theta_k)\\
   	&+\left( \frac{\tilde{C}_0^2\beta L M^2\sqrt{M^2+\epsilon}}{\epsilon C_0^2(1-\beta)^2}+\frac{M^2L\sqrt{M^2+\epsilon}}{2\epsilon(1-\beta)}\right) \sum_{k=1}^{K}\eta_{k}^2.
   	\end{aligned}
\end{equation}
On considering that
\begin{equation}\label{th1.21}
   	\begin{aligned}
    \frac{\tilde{C}_0^2\beta L M^2\sqrt{M^2+\epsilon}}{\epsilon C_0^2(1-\beta)^2}+\frac{M^2L\sqrt{M^2+\epsilon}}{2\epsilon(1-\beta)}
   	&\overset{(o)}{\le}\frac{\tilde{C}_0^2 L M^2\sqrt{M^2+\epsilon}}{\epsilon C_0^2(1-\beta)^2}+\frac{M^2L\sqrt{M^2+\epsilon}C_0^2(1-\beta)}{2\epsilon C_0^2(1-\beta)^2}\\
   	&\overset{(p)}{\le}\frac{\tilde{C}_0^2 L M^2\sqrt{M^2+\epsilon}}{\epsilon C_0^2(1-\beta)^2}+\frac{C_0 M^2L\sqrt{M^2+\epsilon}}{\epsilon C_0^2(1-\beta)^2}\\
   	&\,=\frac{M^2L\sqrt{M^2+\epsilon}(\tilde{C}_0^2+C_0^2)}{\epsilon C_0^2(1-\beta)^2}
   	\overset{(q)}{\le}\frac{2\tilde{C}_0^2 M^2L\sqrt{M^2+\epsilon}}{\epsilon C_0^2(1-\beta)^2},
   	\end{aligned}
\end{equation}
where $(o)$ uses $0\le\beta<1$, $(p)$ holds by $1-\beta<2$, $(q)$ is because $0<C_0\le\tilde{C}_0$. Let
\begin{align}
   	&C_1:=\frac{\sqrt{M^2+\epsilon}\left( f(\mathbf{x}^1)-f^*\right) }{1-\beta},\nonumber\\
   	&C_2:=\frac{\tilde{C}_0 M^4\sqrt{d(M^2+\epsilon)}}{\epsilon^{\frac{3}{2}}C_0(1-\beta)^2},\nonumber\\
   	&C_3:=\frac{2\tilde{C}_0^2 M^2L\sqrt{M^2+\epsilon}}{\epsilon C_0^2(1-\beta)^2}.\nonumber
\end{align}
Substituting \eqref{th1.21} into \eqref{th1.20}, then \eqref{th1.20} becomes
\begin{equation}\label{th1.13}
   	\begin{aligned}
   	\sum_{k=1}^{K}\eta_k\mathbb{E}\left[\left\| \nabla f(\mathbf{x}^k)\right\|^2\right]
   	\le C_1
   	+C_2\sum_{k=1}^{K}\eta_k(1-\theta_k)+C_3 \sum_{k=1}^{K}\eta_{k}^2.
   	\end{aligned}
\end{equation}
Because of
\begin{equation}
   	\begin{aligned}
   	\min_{1\le k\le K}\mathbb{E}\left[\left\| \nabla f(\mathbf{x}^k)\right\|^2\right]\sum_{k=1}^{K}\eta_{k}\le\sum_{k=1}^{K}\eta_k\mathbb{E}\left[\left\| \nabla f(\mathbf{x}^k)\right\|^2\right],\nonumber
   	\end{aligned}
\end{equation}
we can derive
\begin{equation}\label{th1.16}
   	\begin{aligned}
   	\min_{1\le k\le K}\mathbb{E}\left[\left\| \nabla f(\mathbf{x}^k)\right\|^2\right]
   	\le \frac{C_1
   		+C_2\sum_{k=1}^{K}\eta_k(1-\theta_k)+C_3 \sum_{k=1}^{K}\eta_{k}^2}{\sum_{k=1}^{K}\eta_{k}}.\nonumber
   	\end{aligned}
\end{equation}
This completes the proof. 

\subsection{Proof of Theorem \ref{th4}}\label{Proof of th4}
Since Assumptions \ref{Assunbia} and \ref{Assmin} are satisfied, then the equation \eqref{th1.13} in the proof of Theorem \ref{th1} still holds,  that is, 

\begin{equation}
	\begin{aligned}
		\sum_{k=1}^{K}\eta_k\mathbb{E}\left[\left\| \nabla f(\mathbf{x}^k)\right\|^2\right]
		\le C_1
		+C_2\sum_{k=1}^{K}\eta_k(1-\theta_k)+C_3 \sum_{k=1}^{K}\eta_{k}^2,
		\nonumber
	\end{aligned}
\end{equation}
where $C_1:=\frac{\sqrt{M^2+\epsilon}\left( f(\mathbf{x}^1)-f^*\right) }{1-\beta}$, $
C_2:=\frac{\tilde{C}_0 M^4\sqrt{d(M^2+\epsilon)}}{\epsilon^{\frac{3}{2}}C_0(1-\beta)^2}$ and $
C_3:=\frac{2\tilde{C}_0^2 M^2L\sqrt{M^2+\epsilon}}{\epsilon C_0^2(1-\beta)^2}$ are all finite constants. By the condition that $C_{0}\alpha_{k}\leq \eta_{k}\leq \widetilde{C}_{0}\alpha_{k}$, we have
\begin{equation}\label{th4.2}
   	\begin{aligned}
   	\sum_{k=1}^{K}\alpha_k\mathbb{E}\left[\left\| \nabla f(\mathbf{x}^k)\right\|^2\right]
   	&\le\frac{1}{C_0}\sum_{k=1}^{K}\eta_k\mathbb{E}\left[\left\| \nabla f(\mathbf{x}^k)\right\|^2\right]\\
   	&\le \frac{C_1}{C_0}
   	+\frac{C_2}{C_0}\sum_{k=1}^{K}\eta_k{(1-\theta_k)}+\frac{C_3}{C_0} \sum_{k=1}^{K}\eta_{k}^2.
   	\end{aligned}
\end{equation}
Upon the non-increasing of $\alpha_k$ gives
\begin{equation}\label{th4.3}
   	\begin{aligned}
   	\alpha_K\sum_{k=1}^{K}\mathbb{E}\left[\left\| \nabla f(\mathbf{x}^k)\right\|^2\right]
   	\le\sum_{k=1}^{K}\alpha_k\mathbb{E}\left[\left\| \nabla f(\mathbf{x}^k)\right\|^2\right].
   	\end{aligned}
\end{equation}
Recalling the condition that  $\eta_{K}\le\tilde{C}_0\alpha_K$, we arrive at
\begin{equation}\label{th4.5}
   	\begin{aligned}
   	\frac{\eta_K}{\tilde{C}_0}\sum_{k=1}^{K}\mathbb{E}\left[\left\| \nabla f(\mathbf{x}^k)\right\|^2\right]
   	\le\alpha_K\sum_{k=1}^{K}\mathbb{E}\left[\left\| \nabla f(\mathbf{x}^k)\right\|^2\right]
   	.
   	\end{aligned}
\end{equation}
Substituting \eqref{th4.3} and \eqref{th4.5} back into \eqref{th4.2}, we have
\begin{equation}\label{th4.4}
   	\begin{aligned}
   	\frac{\eta_K}{\tilde{C}_0}\sum_{k=1}^{K}\mathbb{E}\left[\left\| \nabla f(\mathbf{x}^k)\right\|^2\right]
   	\le \frac{C_1}{C_0}
   	+\frac{C_2}{C_0}\sum_{k=1}^{K}\eta_k{(1-\theta_k)}+\frac{C_3}{C_0} \sum_{k=1}^{K}\eta_{k}^2.
   	\end{aligned}
\end{equation}
Multiplying both sides of \eqref{th4.4} by $\tilde{C}_0/\eta_K$, we arrive at
\begin{equation}\label{th4.6}
   	\begin{aligned}
   	\sum_{k=1}^{K}\mathbb{E}\left[\left\| \nabla f(\mathbf{x}^k)\right\|^2\right]
   	\le\frac{ \frac{{\tilde{C}_0}C_1}{C_0}
   		+\frac{{\tilde{C}_0}C_2}{C_0}\sum_{k=1}^{K}\eta_k{(1-\theta_k)}+\frac{{\tilde{C}_0}C_3}{C_0} \sum_{k=1}^{K}\eta_{k}^2}{\eta_{K}}.
   	\end{aligned}
\end{equation}
   	Upon dividing both sides of \eqref{th4.6} by $K$ yields
   	\begin{equation}\label{th4.7}
   	\begin{aligned}
   	\frac{1}{K}\sum_{k=1}^{K}\mathbb{E}\left[\left\| \nabla f(\mathbf{x}^k)\right\|^2\right]
   	\le\frac{ \frac{{\tilde{C}_0}C_1}{C_0}
   		+\frac{{\tilde{C}_0}C_2}{C_0}\sum_{k=1}^{K}\eta_k{(1-\theta_k)}+\frac{{\tilde{C}_0}C_3}{C_0} \sum_{k=1}^{K}\eta_{k}^2}{K\eta_{K}}.\nonumber
   	\end{aligned}
\end{equation}
Since $\tau$ is uniformly chosen from the set $\left\lbrace 1,2,\cdots,K\right\rbrace$, we finally obtain
\begin{equation}\label{th4.10}
   	\begin{aligned}
   	&\mathbb{E}\left[\left\| \nabla f(\mathbf{x}^{\tau})\right\|^2\right]
   	=\frac{1}{K}\sum_{k=1}^{K}\mathbb{E}\left[\left\| \nabla f(\mathbf{x}^k)\right\|^2\right]
   	\le\frac{ \frac{{\tilde{C}_0}C_1}{C_0}
   	+\frac{{\tilde{C}_0}C_2}{C_0}\sum_{k=1}^{K}\eta_k{(1-\theta_k)}+\frac{{\tilde{C}_0}C_3}{C_0} \sum_{k=1}^{K}\eta_{k}^2}{K\eta_{K}}.\nonumber
   \end{aligned}
\end{equation}
This completes the proof.

\subsection{Lemma for Theorem \ref{th6}}
\begin{lemma}[Lemma 2 \citet{liu2022almost}]\label{Lema.s.}
Let $\left(  X_k\right)_{k\ge 1}$ be  a sequence of non-negative real numbers and $\left(  a_k\right)_{k\ge 1}$  be a decreasing sequence of positive real numbers such that
\begin{gather}
   	\sum _{k=1}^{\infty}a_{k}X_{k}<\infty,
   	\quad 
   	\sum _{k=1}^{\infty}\frac{a_{k}}{\sum _{i=1}^{k-1}a_{i}}= \infty ,\nonumber
\end{gather}
then we have
\begin{gather} 
   	\min_{1 \le i \le k}X_{i}=o\left( \frac{1}{\sum _{i=1}^{k}a_{i}}\right) .\nonumber
\end{gather}
\end{lemma}

\subsection{Proof of Theorem \ref{th6}}\label{Proof of th6}

From the the $L$-smoothness of $f$ and Assumption \ref{Assunbia}, it is clear that \eqref{th1.13} in the proof of Theorem \ref{th1} still holds, that is,
\begin{equation}
   	\begin{aligned}
   	\sum_{k=1}^{K}\eta_k\mathbb{E}\left[\left\| \nabla f(\mathbf{x}^k)\right\|^2\right]
   	\le C_1
   	+C_2\sum_{k=1}^{K}\eta_k(1-\theta_k)+C_3 \sum_{k=1}^{K}\eta_{k}^2
   	,\nonumber
   	\end{aligned}
\end{equation}
where $C_1:=\frac{\sqrt{M^2+\epsilon}\left( f(\mathbf{x}^1)-f^*\right) }{1-\beta}$, $C_2:=\frac{\tilde{C}_0 M^4\sqrt{d(M^2+\epsilon)}}{\epsilon^{\frac{3}{2}}C_0(1-\beta)^2}$ and $C_3:=\frac{2\tilde{C}_0^2 M^2L\sqrt{M^2+\epsilon}}{\epsilon C_0^2(1-\beta)^2}$ are all finite values.
By the conditions that $\sum_{k=1}^{\infty}\eta_{k}^2<\infty$ and $\sum_{k=1}^{\infty}\eta_{k}(1-\theta_k)<\infty$, we have
\begin{equation}\label{th2.2}
   	\begin{aligned}
   	\sum_{k=1}^{\infty}\eta_k\mathbb{E}\left[\left\| \nabla f(\mathbf{x}^k)\right\|^2\right]
   	\le C_1
   	+C_2\sum_{k=1}^{\infty}\eta_k(1-\theta_k)+C_3 \sum_{k=1}^{\infty}\eta_{k}^2
   	<\infty.\nonumber
   	\end{aligned}
\end{equation}
Upon applying the Monotone Convergence Theorem \citep[see][chap.~2 Theorem 14]{galambos1995advanced} gives
\begin{equation}
   	\begin{aligned}
   	\mathbb{E}\left[\sum_{k=1}^{\infty}\eta_k\left\| \nabla f(\mathbf{x}^k)\right\|^2\right]
   	=\sum_{k=1}^{\infty}\eta_k\mathbb{E}\left[\left\| \nabla f(\mathbf{x}^k)\right\|^2\right]
   	<\infty.\nonumber
   	\end{aligned}
\end{equation}
Then we have
\begin{equation}\label{th6.1}
   	\begin{aligned}
   	\sum_{k=1}^{\infty}\eta_k\left\| \nabla f(\mathbf{x}^k)\right\|^2<\infty\ a.s.
   	\end{aligned}
\end{equation}
Let $a_k:=\eta_{k}$ and $X_k:=\left\| \nabla f(\mathbf{x}^k)\right\|^2$, then combine the condition that $\sum_{k=1}^{\infty}\big( {\eta_k}/\big( {\sum_{i=1}^{k-1}\eta_i}\big)\big) =\infty$ and \eqref{th6.1}, we have
\begin{equation}
   	\begin{aligned}
   	\sum_{k=1}^{\infty}a_k X_k<\infty\ a.s.,\quad
   	\sum_{k=1}^{\infty}\frac{a_k}{\sum_{i=1}^{k-1}a_i}=\sum_{k=1}^{\infty}\frac{\eta_k}{\sum_{i=1}^{k-1}\eta_i}=\infty.\nonumber
   	\end{aligned}
\end{equation}
Thus, by using Lemma \ref{Lema.s.}, we get
\begin{align}
   	\min_{1\le k\le K}\left\| \nabla f(\mathbf{x}^k)\right\|^2=o\left( \frac{1}{\sum_{k=1}^{K}\eta_{k}}\right) \ a.s. \nonumber
\end{align}
The proof is completed.
   	
\subsection{Proof of Corollary \ref{cor4}}\label{Proof of cor4}

First, we verify that $\eta_{k}$ and $\theta_{k}$ satisfy the conditions in Theorem \ref{th6}. Since $\eta_{k}=1/k^{q}$ for any $1/2<q<1$, we then obtain
\begin{equation}
   	\begin{aligned}
     	\label{cor4.1}&\sum _{i=1}^{k-1}\eta _{i}= \sum _{i=1}^{k-1}\frac{1}{i^{q}}\leq 1 + \int _{1}^{k}\frac{1}{x^{q}}dx= 1 + \frac{k^{1-q}-1}{1-q}=\frac{k^{1-q}-q}{1-q}\leq \frac{k^{1-q}}{1-q}.\nonumber
   	\end{aligned}
\end{equation}
Thus, we have 
\begin{align}
   		\label{en.3}&\sum_{k=1}^{\infty}\frac{\eta_k}{\sum_{i=1}^{k-1}\eta_i}
   		\ge\sum_{k=1}^{\infty}\frac{1}{k^q}\frac{1-q}{k^{1-q}}
   		=\sum_{k=1}^{\infty}\frac{1-q}{k}
   		=\infty,\\
   		\label{en.4}&\sum_{k=1}^{\infty}\eta_{k}^2
   		=\sum_{k=1}^{\infty}\frac{1}{k^{2q}}
   		\overset{(a)}{<}\infty,\\
   		\label{en.5}&\sum_{k=1}^{\infty}\eta_{k}(1-\theta_k)
   		=\sum_{k=1}^{\infty}\frac{1}{k^{p+q}}
   		\overset{(b)}{<}\infty,
\end{align}
where $(a)$ uses $q>1/2$ and $(b)$ uses $p>1-q$,  that is,  $p+q>1$. From \eqref{en.3}, \eqref{en.4} and  \eqref{en.5}, we know that $\eta_{k}$ and $\theta_k$ satisfy the conditions in \eqref{th6.2}. Thus, by using the conclusion of Theorem \ref{th6}, we have
\begin{align}\label{cor4.3}
   	\min_{1\le k\le K}\left\| \nabla f(\mathbf{x}^k)\right\|^2=o\left( \frac{1}{\sum_{k=1}^{K}\eta_{k}}\right)\ a.s. 
\end{align}
Using the relationship between integration and summation yields
\begin{align}\label{cor4.2}
   	&\sum _{k=1}^{K}\eta _{k}
   	= \sum _{k=1}^{K}\frac{1}{k^{q}}\geq \int _{1}^{K}\frac{1}{x^{q}}dx= \frac{K^{1-q}-1}{1-q}.
\end{align}
Substituting \eqref{cor4.2} into \eqref{cor4.3} to get
\begin{equation}
   		\begin{aligned}
   		\min_{1\le k\le K}\left\| \nabla f(\mathbf{x}^k)\right\|^2
   		=o\left( \frac{1}{K^{1-q}}\right)  \ a.s. \nonumber
   		\end{aligned}
\end{equation}
The proof is completed.

\section{Proofs for Section \ref{Non-ergodic Convergence of Adam}}
\subsection{Lemma for Theorem \ref{th2}}
\begin{lemma}[Corollary A.1 \citet{Liu2021ConvergenceAO}]
\label{Lemab} 
Let $\left(  a_k\right)_{k\ge 1} $ and $\left(  b_k\right)_{k\ge 1}$ be the sequences of non-negative real values satisfying $\sum_{k=1}^{\infty}a_k=\infty,\sum_{k=1}^{\infty}a_k b_k^2<\infty$. If there exists a constant $c>0$ such that $|b_{k+1}-b_{k}|\le c a_k$, then we have
\begin{gather} 
   	\lim_{k\rightarrow\infty}b_k=0.\nonumber
\end{gather}
\end{lemma}
   
\subsection{Proof of Theorem \ref{th2}}\label{Proof of th2}

First, we prove the conclusion that $\lim_{k\rightarrow\infty} \left\| \nabla f(\mathbf{x}^k)\right\|=0\ a.s.$ Since $f$ is $L$-smooth and Assumption \ref{Assunbia} is satisfied, then the equation \eqref{th6.1} in the proof of Theorem \ref{th6} still holds, that is,
\begin{equation}
   	\begin{aligned}
   	\sum_{k=1}^{\infty}\eta_k\left\| \nabla f(\mathbf{x}^k)\right\|^2<\infty\ a.s.\nonumber
   	\end{aligned}
\end{equation}
Using the $L$-smoothness of $f$ again, we arrive at
\begin{equation}
   	\begin{aligned}
   	\left| \left\| \nabla f(\mathbf{x}^{k+1})\right\|-\left\| \nabla f(\mathbf{x}^{k})\right\| \right| 
   	\le&\left\| \nabla f(\mathbf{x}^{k+1})-\nabla f(\mathbf{x}^{k})\right\|
   	\overset{(a)}{\le} L\left\| \mathbf{x}^{k+1}-\mathbf{x}^{k}\right\| \\
   	=& L\eta_{k}\left\|  \frac{\mathbf{m}^k}{\sqrt{\mathbf{v}^k+\epsilon}} \right\|   
   	\overset{(b)}{\le} \frac{LM}{\sqrt{\epsilon}}\eta_{k}\ a.s.,\nonumber
   	\end{aligned}
\end{equation}
where $(a)$ uses Assumption \ref{Assmin}, $(b)$ holds by Lemma \ref{Lemmbound}, that is, $\| {\mathbf{m}^k}/{\sqrt{\mathbf{v}^k+\epsilon}}\| \le{M}/{\sqrt{\epsilon}}\ a.s.$ 
Next, we denote that
\begin{align*}
   	a_k=\eta_{k},\quad b_k= \left\| \nabla f(\mathbf{x}^{k})\right\|,\quad c=\frac{LM}{\sqrt{\epsilon}}.
\end{align*}
Thus, we can get 
\begin{align*}
   	&\sum_{k=1}^{\infty}a_k
   	=\sum_{k=1}^{\infty}\eta_{k}
   	=\infty, 
   	\quad 
   	\sum_{k=1}^{\infty}a_k b_k^2
   	=\sum_{k=1}^{\infty}\eta_k\left\| \nabla f(\mathbf{x}^k)\right\|^2<\infty\ a.s.,\\
   	&\left| b_{k+1}-b_{k}\right| 
   	=\left|\left\| \nabla f(\mathbf{x}^{k+1})\right\|-\left\| \nabla f(\mathbf{x}^{k})\right\| \right| 
   	\le \frac{LM}{\sqrt{\epsilon}}\eta_{k}
   	=c a_{k}
   	\ a.s.,
\end{align*}
which matches the condition of Lemma \ref{Lemab}. Then from the conclusion of Lemma \ref{Lemab}, we have
\begin{equation}
   	\begin{aligned}
   	\lim_{k\rightarrow\infty} \left\| \nabla f(\mathbf{x}^k)\right\|=0\ a.s.\nonumber
   	\end{aligned}
\end{equation}
Further, since we obtain $\left\| \nabla f(\mathbf{x}^k)\right\|\le M\ a.s.$ in Lemma \ref{Lemmbound}, then the Dominated Convergence Theorem \citep[see][chap.~2 Theorem 16]{galambos1995advanced} gives
\begin{equation}
   	\begin{aligned}
   	\lim_{k\rightarrow\infty} \mathbb{E}\left[\left\| \nabla f(\mathbf{x}^k)\right\|\right]=0.\nonumber
   	\end{aligned}
\end{equation}
The proof is completed.

\subsection{Proof of Theorem \ref{th5}}\label{Proof of th5}
Because the conditions in Theorem \ref{th2} still hold, we begin with the conclusion of Theorem \ref{th2},  that is, 
\begin{equation}
   	\begin{aligned}
   	\lim_{k\rightarrow\infty} \left\| \nabla f(\mathbf{x}^k)\right\|=0\ a.s.,\ 
   	\text{and}\ \lim_{k\rightarrow\infty} \mathbb{E}\left[\left\| \nabla f(\mathbf{x}^k)\right\|\right]=0.\nonumber
   	\end{aligned}
\end{equation}
Since $f$ satisfies the PL condition,  that is,  for $\forall\mathbf{x}\in{\mathbb{R}^d}$, there exists $v>0$, such that $\left\| \nabla f(\mathbf{x})\right\|  ^2\ge2v(f(\mathbf{x})-f^*)$, we then obtain
\begin{align}\label{th5.1}
   	\left\| \nabla f(\mathbf{x}^k)\right\|^2
   	\ge 2v(f(\mathbf{x}^k)-f^*)
   	\ge 0.
\end{align}
Taking the limit on both sides of \eqref{th5.1} with respect to $k$ and using the Squeeze theorem, we can get
\begin{equation}
   	\begin{aligned}
   	\lim_{k\rightarrow\infty} f(\mathbf{x}^k)=f^*\ a.s.\nonumber
   	\end{aligned}
\end{equation}
On the other hand, taking the total expectation on \eqref{th5.1} gives
\begin{align}\label{en.6}
   	\mathbb{E}\left[ \left\| \nabla f(\mathbf{x}^k)\right\|^2\right] 
   	\ge 2v\left( \mathbb{E}\left[ f(\mathbf{x}^k)\right] -f^*\right) 
   	\ge 0.
\end{align}
Upon taking the limit on both sides of \eqref{en.6} and combining it with the Squeeze theorem, we have
\begin{equation}
   	\begin{aligned}
   	\lim_{k\rightarrow\infty} \mathbb{E}\left[f(\mathbf{x}^k)\right]=f^*.\nonumber
   	\end{aligned}
\end{equation}
This completes the proof.

\subsection{Lemma for Theorem \ref{th3}}
\begin{lemma}\label{Lemre}
Let $\gamma_k\in\left( 0,1 \right]$ for $k=1,2,\cdots$ and the sequence $\left(  \Delta_k\right)_{k\ge 1}$ satisfy
\begin{equation}\label{en.2}
   	\begin{aligned}
   	\Delta_{k+1}\le (1-\gamma_k)\Delta_{k}+B_k,\ k\ge 1,
   	\end{aligned}
\end{equation}
then for any $K\ge 1$, we have 
\begin{equation}\label{Lemre.1}
   	\begin{aligned}
   	\Delta_{K+1}\le \Gamma_K(1-\gamma_1)\Delta_1+\Gamma_K\sum_{k=1}^{K}\frac{B_k}{\Gamma_k},
   	\end{aligned}
\end{equation}
where
\begin{equation}
   	\begin{aligned}
   	\Gamma_k:=
   	\begin{cases}
   	1,& \text{$k=1$}\\
   	(1-\gamma_k)\Gamma_{k-1},& \text{$k\ge2$}
   	\end{cases}.\nonumber
   	\end{aligned}
\end{equation}
\end{lemma}
\begin{proof}
It follows from the condition in \eqref{en.2} that
\begin{equation}
   	\begin{aligned}
   	\Delta_{K+1}
   	\le& (1-\gamma_K)\Delta_{K}+B_K
   	\\
   	\le&(1-\gamma_K)(1-\gamma_{K-1})\Delta_{K-1}+(1-\gamma_K){B_{K-1}}+{B_{K}}\\
   	\le&\cdots\le\prod_{k=1}^{K}(1-\gamma_k)\Delta_{1}+\sum_{k=1}^{K}\left( \prod_{i=k+1}^{K}(1-\gamma_i)\right) B_{k}
   	\\
   	{=}&\prod_{k=2}^{K}(1-\gamma_k)(1-\gamma_1)\Delta_{1}
   	+\sum_{k=1}^{K}\left( \prod_{i=2}^{K}(1-\gamma_i)\bigg/\prod_{i=2}^{k}(1-\gamma_i)\right) B_{k}\\
   	=&\prod_{k=2}^{K}(1-\gamma_k)(1-\gamma_1)\Delta_{1}
   	+\prod_{i=2}^{K}(1-\gamma_i)\sum_{k=1}^{K}
   	\left(B_{k}\bigg/\prod_{i=2}^{k}(1-\gamma_i)\right) \\
   	=&\Gamma_K (1-\gamma_{1})\Delta_{1}+\Gamma_K\sum_{k=1}^{K}\frac{B_k}{\Gamma_k}.\nonumber
   	\end{aligned}
\end{equation}
The proof is completed.
\end{proof}

\begin{remark}\label{RMKLemre}
It should be noted that Lemma \ref{Lemre} plays an important role in establishing the proof of Theorem \ref{th3}. Lemma \ref{Lemre} is inspired by \citet[Lemma 1]{ghadimi2016accelerated}, with difference recursion about $\gamma_{k}$, $\Delta_{k}$ and $B_{k}$ in \eqref{en.2}. In particular, by choosing $\gamma_k={2}/{(k+1)}$, we can obtain that $\Gamma_k={2}/{(k(k+1))}$ and $\Gamma_k(1-\gamma_{1})\Delta_1=0$. Thus, the conclusion of Lemma \ref{Lemre} becomes $\Delta_{K+1}
\le\Gamma_K\sum_{k=1}^{K}
({B_k}/{\Gamma_k})=\big(\sum_{k=1}^{K}k(k+1){B_k}\big)/(K(K+1))$.
\end{remark}

\subsection{Proof of Theorem \ref{th3}}\label{Proof of th3}
From Assumptions \ref{Assunbia} and \ref{Assmin}, and the step size condition, the equation \eqref{th1.14} in the proof of Theorem \ref{th1} still holds. Then combining \eqref{th1.14} with the PL condition gives
\begin{equation}\label{7.1}
   	\begin{aligned}
   	\mathbb{E}\left[ f(\mathbf{x}^{k+1})\right] 
   	&{\le}
   	\mathbb{E}\left[ f(\mathbf{x}^k)\right] 
   	-\frac{1-\beta}{\sqrt{M^2+\epsilon}}\eta_k\mathbb{E}\left[\left\| \nabla f(\mathbf{x}^k)\right\|^2\right]\\
   	&{+}\frac{\beta L M^2}{\epsilon}\eta_k\sum_{j=1}^{k}\beta^{k-j} \eta_{j-1}
   	+\frac{\sqrt{d} M^4}{\epsilon^{\frac{3}{2}}}\eta_k\sum_{j=1}^{k}\beta^{k-j}(1-\theta_j)
   	+\frac{M^2L}{2\epsilon}\eta_k^2 \\
   	&{\overset{\text{PL}}{\le}}
   	\mathbb{E}\left[ f(\mathbf{x}^k)\right] 
   	-\frac{2(1-\beta)v}{\sqrt{M^2+\epsilon}}\eta_k\left( \mathbb{E}\left[ f(\mathbf{x}^k)\right] -f^*\right) \\
   	&{+}\frac{\beta LM^2}{\epsilon}
   	\eta_k\sum_{j=1}^{k}\beta^{k-j} \eta_{j-1}
   	+\frac{\sqrt{d} M^4}{\epsilon^{\frac{3}{2}}}\eta_k\sum_{j=1}^{k}\beta^{k-j}(1-\theta_j)
   	+\frac{M^2L}{2\epsilon}\eta_k^2.
   	\end{aligned}
\end{equation}
Subtracting $f^*$ from both sides of \eqref{7.1} and substituting $\eta_{k}=\frac{\sqrt{M^2+\epsilon}}{(1-\beta)v}\cdot\frac{1}{k+1}$ into the above equation, we obtain
\begin{equation}\label{th3.1}
   	\begin{aligned}
   	\mathbb{E}\left[ f(\mathbf{x}^{k+1})\right] -f^*
   	\le&\left( 1-\frac{2}{k+1}\right) \left( \mathbb{E}\left[ f(\mathbf{x}^k)\right] -f^*\right) 
   	+\underbrace{\frac{\beta LM^2(M^2+\epsilon)}{(1-\beta)^2 v^2\epsilon}}_{C_4}\cdot{\frac{1}{k+1}\sum_{j=1}^{k}\beta^{k-j} \frac{1}{j}}\\
   	&+\underbrace{\frac{M^4\sqrt{d\left( M^2+\epsilon\right)}  }{(1-\beta)v\epsilon^{\frac{3}{2}}}}_{C_5}\cdot{\frac{1}{k+1}\sum_{j=1}^{k}\beta^{k-j}(1-\theta_j)}
   	+\underbrace{{\frac{LM^2(M^2+\epsilon)}{2(1-\beta)^2 v^2\epsilon}}}_{C_6}\cdot{\frac{1}{(k+1)^2}}
   	.
   	\end{aligned}
\end{equation}
Now, let's set
\begin{align*}
   	&\Delta_k:=\mathbb{E}\left[ f(\mathbf{x}^k)\right] -f^*,
   	\quad 
   	\gamma_k:=\frac{2}{k+1},\\
   	&B_k:=C_4\cdot {\frac{1}{k+1}\sum_{j=1}^{k}\beta^{k-j} \frac{1}{j}}
   	+C_5\cdot {\frac{1}{k+1}\sum_{j=1}^{k}\beta^{k-j}(1-\theta_j)}
   	+C_6\cdot {\frac{1}{(k+1)^2}} ,\ k\ge 1.
\end{align*}
Then the equation \eqref{th3.1} becomes
\begin{align*}
   	\Delta_{k+1}\le (1-\gamma_k)\Delta_{k}+B_k,\ k\ge 1,
\end{align*}
which satisfies the condition of Lemma \ref{Lemre}, and the sequence $\gamma_k$ meets the condition discussed in Remark \ref{RMKLemre},  that is,  if $\gamma_k={2}/{(k+1)}$, then  $\Delta_{K+1}\le(\sum_{k=1}^{K}k(k+1){B_k})/(K(K+1))$. So,   applying Lemma \ref{Lemre} and Remark \ref{RMKLemre} in conjunction with \eqref{th3.1} yields
\begin{equation}\label{th3.5}
   	\begin{aligned}
   	&\mathbb{E}\left[ f(\mathbf{x}^{K+1})\right] -f^*
   	\le \frac{1}{K(K+1)}\sum_{k=1}^{K}k(k+1)B_k\\
   	&=\frac{1}{K(K+1)}  \left( C_4 \sum_{k=1}^{K}k(k+1){\frac{1}{k+1}\sum_{j=1}^{k}\beta^{k-j} \frac{1}{j}}
   	\right.\\
   	&\left.+C_5\sum_{k=1}^{K}k(k+1){\frac{1}{k+1}\sum_{j=1}^{k}\beta^{k-j}(1-\theta_j)}
   	+C_6\sum_{k=1}^{K}k(k+1){\frac{1}{(k+1)^2}}\right) \\
   	&=\frac{1}{K(K+1)}  \left( C_4 \sum_{k=1}^{K}k{\sum_{j=1}^{k}\beta^{k-j} \frac{1}{j}}
   	+C_5\sum_{k=1}^{K}k{\sum_{j=1}^{k}\beta^{k-j}(1-\theta_j)}
   	+C_6\sum_{k=1}^{K}{\frac{k}{k+1}}\right) .
   	\end{aligned}
\end{equation}
Further, we estimate the three terms in the right-hand side of \eqref{th3.5} as follows.
\begin{align}\label{th3.2}
   	&\sum_{k=1}^{K}k\sum_{j=1}^{k}\beta^{k-j}\frac{1}{j}
   	\overset{(b)}{\le}\frac{1}{(1-\beta)^2}\sum_{k=1}^{K}k\cdot\frac{1}{k}
   	=\frac{1}{(1-\beta)^2} K,\\
   	\label{th3.4}
   	&\sum_{k=1}^{K}k{\sum_{j=1}^{k}\beta^{k-j}(1-\theta_j)}
   	\overset{(c)}{\le}\frac{1}{(1-\beta)^2}\sum_{k=1}^{K}k(1-\theta_k)
   	,\\
   	\label{th3.3}
   	&\sum_{k=1}^{K}\frac{k}{k+1}
   	{\le}\sum_{k=1}^{K} 1
   	=K,
\end{align}
where inequalities $(b)$ and $(c)$ uses  \eqref{Lemsum.2} in Lemma \ref{Lemsum}, that is,  $\sum_{k=1}^{K}k\sum_{j=1}^{k}\beta^{k-j}b_j\le \big(1/(1-\beta)^{2}\big)\sum_{k=1}^{K}k b_k$ with $b_j:=1/j$ in $(b)$, $b_j:=(1-\theta_j)$ in $(c)$. Upon substituting \eqref{th3.2}, \eqref{th3.3} and \eqref{th3.4} into \eqref{th3.5}, we have
\begin{equation}
   	\begin{aligned}
   	\mathbb{E}\left[ f(\mathbf{x}^{K+1})\right] -f^*
   	\le&\frac{1}{K(K+1)} \left( \frac{C_4}{(1-\beta)^2} K+\frac{C_5}{(1-\beta)^2}\sum_{k=1}^{K}k(1-\theta_k)+C_6  K\right) \\
   	=&\left( \frac{C_4}{(1-\beta)^2}+C_6\right) \cdot\frac{ 1 }{K+1} +\frac{C_5}{(1-\beta)^2}\cdot\frac{\sum_{k=1}^{K}k(1-\theta_k)}{K(K+1)}.\nonumber
   	\end{aligned}
\end{equation}
The proof is completed.

\section{Proof of Proposotion \ref{prop-non}}\label{A}

\begin{proof}
By the condition that $\lim_{n\to\infty}x_{n}=x$, we can obtain that $\lim_{n\to\infty}x_{n}-x=0$, which implies that there exists a constant $G>0$ such that $\lvert x_n-x\rvert< G$ holds for $\forall\,n\in \mathbb{N}_{+}$ and also suggest that for $\forall\,\epsilon>0$, $\exists\, N_1\in \mathbb{N}_{+}$, such that $\lvert x_n-x\rvert<\epsilon/2$ holds for $\forall\,n>N_1$. Upon since $\lim_{n\rightarrow\infty} \omega_{n,k}=0$ for any $1\leq k\leq n$, we can derive that $\exists\, N_2\in \mathbb{N}_{+}$, such that $|\omega_{n,k}|<\epsilon/(2GN_1)$ holds for $\forall\,n>N_2$ and $\forall\,k\le N_1$. Let $N=\max\left\lbrace N_1,N_2\right\rbrace $, then for $\forall\,n>N$, we have
\begin{equation}
\begin{aligned}\label{barx}
	\lvert \bar{x}_n-x\rvert 
	&=\left\lvert \sum_{k=1}^{n}\omega_{n,k}x_k
	-x\right\rvert 
	\overset{(a)}{=}
	\left\lvert \sum_{k=1}^{n}\omega_{n,k}
	(x_k-x)\right\rvert
	\\
	&\le
	\left\lvert 
	\sum_{k=1}^{N_1}\omega_{n,k}(x_k-x)
	\right\rvert
	+\left\lvert \sum_{k=N_{1}+1}^{n}\omega_{n,k}(x_k-x)
	\right\rvert
	\\
	&\le \sum_{k=1}^{N_1}\omega_{n,k}
	\lvert x_k-x\rvert
	+ \sum_{k=N_{1}+1}^{n}\omega_{n,k}
	\lvert x_k-x\rvert
	\\
	&<G\sum_{k=1}^{N_1}\omega_{n,k}
	+\frac{\epsilon}{2}\sum_{k=N_{1}+1}^{n}\omega_{n,k}
	\\
	&< GN_1
	\,\frac{\epsilon}{2GN_1}
	+\frac{\epsilon}{2}\sum_{k=1}^{n}\omega_{n,k}
	\overset{(b)}{=}\frac{\epsilon}{2}+\frac{\epsilon}{2}=\epsilon,
\end{aligned}
\end{equation}
where $(a)$ and $(b)$ is due to $\sum_{k=1}^{n}\omega_{n,k}=1$. Thus, we can conclude from \eqref{barx} that for $\forall\,\epsilon>0$, there exists $N\in\mathbb{N}_{+}$ such that $\lvert \bar{x}_{n}-x\rvert<\epsilon$, which implies that $\lim_{n\to\infty}\bar{x}_{n}=x$. This completes the proof.
\end{proof}

\section{The Relationship between \textbf{SC-Zou} and \textbf{SC-Adam}}\label{B}
In this section, we present the relationship between two sufficient conditions for the convergence of Adam. For clarity, we first list \textbf{SC-Zou} proposed by \citet{zou2019sufficient} and our \textbf{SC-Adam} as follows, then discuss their relationship in Proposition \ref{contopro}.\\[1.5mm]
$1.$ \textbf{SC-Zou}, the sufficient condition in \citet{zou2019sufficient}:
\begin{enumerate}
\item[\textcircled{\scriptsize z1}] $0\le\beta_k\le\beta<1$;
\item[\textcircled{\scriptsize z2}]$0<\theta_k<1$ and $\theta_k$ is monotonically non-decreasing;
\item[\textcircled{\scriptsize z3}] $ \alpha_k\le\eta_k/\sqrt{1-\theta_k}\le C\alpha_k$, where $(\alpha_k)_{k\ge 1}$ is a non-increasing real sequence  and $C>0$;
\item[\textcircled{\scriptsize z4}] $\left( \sum_{k=1}^{K}\eta_k\sqrt{1-\theta_k}\,
\right)\big/\left( K\eta_{K}\right) =o(1)$. 
\end{enumerate}
$2.$ \textbf{SC-Adam}, the sufficient condition in Corollary \ref{suffcor}:
\begin{enumerate}
\item[\textcircled{\small{1}}] $0\le\beta_k\le\beta<1$;
\item[\textcircled{\small{2}}] $0<\theta_k<1$; 
\item[\textcircled{\small{3}}] $C_0\alpha_k\le \eta_{k}\le \tilde{C}_0 \alpha_k$, where $(\alpha_k)_{k\ge 1}$ is a non-increasing real sequence and $C_0$, $\tilde{C}_0>0$;  
\item[\textcircled{\small{4}}] $\left( \sum_{k=1}^{K}\eta_k{(1-\theta_k)}\right) \big/\left( K\eta_{K}\right) =o(1)$;
\item[\textcircled{\small{5}}] $\left( \sum_{k=1}^{K}\eta_k^2\right) \big/\left( K\eta_{K}\right) =o(1)$. 
\end{enumerate}

\begin{proposition}\label{contopro}
The sufficient condition (\textbf{SC-Adam}) obtained in Corollary \ref{suffcor} is weaker than \textbf{SC-Zou} given by \citet{zou2019sufficient}.
\end{proposition}
\begin{proof}
First, it is clear that \textcircled{\scriptsize {z1}} in \textbf{SC-Zou} is the same as \text{\textcircled{\small{1}}} in \textbf{SC-Adam}. Meanwhile, since  \text{\textcircled{\small{2}}} in \textbf{SC-Adam} does not require the non-decreasing of $\theta_{k}$, then it is much weaker than \textcircled{\scriptsize {z2}} in \textbf{SC-Zou}. 
\par Next, we discuss the relationship between \text{\textcircled{\small{3}}}, \text{\textcircled{\small{4}}}, \text{\textcircled{\small{5}}} in \textbf{SC-Adam} and \textcircled{\scriptsize {z2}}, \textcircled{\scriptsize {z3}}, \textcircled{\scriptsize {z4}} in \textbf{SC-Zou} as follows.
\begin{enumerate}
\item[$(1)$] The conditions \textcircled{\scriptsize {z2}} and \textcircled{\scriptsize {z3}} in \textbf{SC-Zou} imply \text{\textcircled{\small{3}}} in \textbf{SC-Adam}.
\par Since there exists a non-increasing sequence $\alpha_k$ and a positive number $C$ such that
$\alpha_k\le\eta_k/\sqrt{1-\theta_k}\le C \alpha_k$, then we can get
\begin{equation}\label{pro.9}
	\begin{aligned}
		\alpha_k \sqrt{1-\theta_k}\le\eta_k\le C \alpha_k \sqrt{1-\theta_k}.
	\end{aligned}
\end{equation}
By the non-decreasing of $\theta_k$ and $0<\theta_k<1$, we obtain that $\sqrt{1-\theta_k}$ is also non-increasing and non-negative. Upon since $\alpha_k$ is non-increasing and non-negative, we then have $\alpha_k \sqrt{1-\theta_k}$ is non-increasing. Thus, \text{\textcircled{\small{$3$}}} is satisifed with $C_0:=1$ and $\tilde{C}_0:=C$.
\item[$(2)$] The conditions \textcircled{\scriptsize {z2}} and \textcircled{\scriptsize {z4}} in \textbf{SC-Zou} imply \text{\textcircled{\small{4}}} in \textbf{SC-Adam}.
\par  By the fact that $1-\theta_k\le\sqrt{1-\theta_k}$ holds for $0<\theta_k<1$, we have
\begin{equation}\label{pro.10}
	\begin{aligned}
		0
		\leq \frac{\sum_{k=1}^{K}\eta_k{\left( 1-\theta_k\right) }}{K\eta_{K}}
		\le\frac{\sum_{k=1}^{K}\eta_k\sqrt{1-\theta_k}}{K\eta_{K}}.
		\nonumber
	\end{aligned}
\end{equation}
Therefore, we can derive $\big(\sum_{k=1}^{K}\eta_k{\left( 1-\theta_k\right)}\big)/(K\eta_{K})=o(1)$ in \text{\textcircled{\small{4}}} from the condition $\big(\sum_{k=1}^{K}\eta_k\sqrt{1-\theta_k}\big)/(K\eta_{K})=o(1)$ in \textcircled{\scriptsize {z4}}.
\item[$(3)$] The conditions \textcircled{\scriptsize {z3}} and \textcircled{\scriptsize {z4}} in \textbf{SC-Zou} imply \text{\textcircled{\small{5}}} in \textbf{SC-Adam}.
\par It follows from \textcircled{\scriptsize {z3}} that
\begin{equation}\label{pro.11}
	\begin{aligned}
		\eta_k\le C \alpha_k \sqrt{1-\theta_k}
		\le C \alpha_1 \sqrt{1-\theta_k},\nonumber
	\end{aligned}
\end{equation}
where the second inequality uses the fact that $\alpha_{k}$ is non-increasing,  that is,  $\alpha_{k}\le\alpha_1$. Further, we can get
\begin{equation}\label{pro.12}
	\begin{aligned}
		0\leq 
		\frac{\sum_{k=1}^{K}\eta_k^2}{K\eta_{K}}\le\frac{C \alpha_1\sum_{k=1}^{K}\eta_k\sqrt{1-\theta_k}}{K\eta_{K}}.
		\nonumber
	\end{aligned}
\end{equation}
Therefore, $\big( \sum_{k=1}^{K}\eta_k\sqrt{1-\theta_k}\,
\big)/( K\eta_{K})=o(1)$ in \textcircled{\scriptsize {z4}} indicates $\big( \sum_{k=1}^{K}\eta_k^2\big)/( K\eta_{K}) =o(1)$ in \text{\textcircled{\small{5}}}.
\end{enumerate}
The proof is completed.
\end{proof}

\vskip 0.2in
\bibliography{sample}

\end{document}